\documentclass[12pt]{amsart}

\usepackage{amssymb,latexsym}
\usepackage{color}
\usepackage{comment}
\usepackage{graphicx}
\usepackage{enumerate}
\usepackage{bbm, dsfont}
\makeatletter
\usepackage{caption}
\@namedef{subjclassname@2010}{

  \textup{2010} Mathematics Subject Classification}
\makeatother
\newtheorem{thm}{Theorem}[section]

\newtheorem{cor}[thm]{Corollary}
\newtheorem{lem}[thm]{Lemma}
\newtheorem{pro}[thm]{Proposition}
\theoremstyle{definition}

\newtheorem{rem}[thm]{Remark}

\numberwithin{equation}{section}
\newtheoremstyle{case}{}{}{}{}{}{:}{ }{}
\theoremstyle{case}
\newtheorem{case}{\textbf{Case}}

\newcommand{\X}{\mathbb{X}}

\newcommand{\ex}{\mathbb{E}}

\newcommand{\Y}{\mathbb{Y}}
\newcommand\Item[1][]{%
  \ifx\relax#1\relax  \item \else \item[#1] \fi
  \abovedisplayskip=0pt\abovedisplayshortskip=0pt~\vspace*{-\baselineskip}}

\begin{document}

\baselineskip=16pt

\title[On Fekete polynomials]{$L_q$ norms and Mahler measure of Fekete polynomials}

\author{Oleksiy Klurman}
\address{School of Mathematics, University of Bristol, UK}

\email{oleksiy.klurman@bristol.ac.uk}

\author{Youness Lamzouri}

\address{Institut \'Elie Cartan de Lorraine, Universit\'e de Lorraine, BP 70239, 54506 Vandoeuvre-l\`es-Nancy Cedex, France}

\email{youness.lamzouri@univ-lorraine.fr}

\author{Marc Munsch}
\address{Università degli Studi di Genova - Dipartimento di Matematican, Via Dodecaneso, 35, 16146 Genova}
\email{munsch@dima.unige.it}

\date{}


%\begin{abstract} We study the number of real zeros of Fekete polynomials, a deterministic family of Littlewood polynomials with coefficients given by real Dirichlet characters. We come close to establish a conjecture of Baker-Montgomery about the average number of real zeros of Fekete polynomials. Additionally, we obtain several localization results in subintervals of $(0,1)$ concerning the real zeros of such polynomials.
%\end{abstract} 

\maketitle

\begin{abstract}
We show that the distribution of the values of Fekete polynomials $F_p$ on the unit circle is governed, as $p\to\infty$,  by an explicit limiting (non-Gaussian) random point process.\\ This allows us to prove that the Mahler measure of $F_p$ satisfies  
$$M_0(F_p)\sim k_0\sqrt{p},$$
as  $p\to\infty$ where $k_0=0.74083\dots,$
thus solving an old open problem.\\
Further, we obtain an asymptotic formula for all moments $\|F_p\|_q$ with $0<q<\infty,$ resolving another open problem and improving previous results of G\"{u}nther and Schmidt (who treated the case $q=2k,$ $k\in\mathbb{N}$).\\
\end{abstract}

%\tableofcontents
\section{Introduction}
Erd\H{o}s and Littlewood explored various extremal problems concerning polynomials with coefficients $\pm 1$ (see in particular Littlewood’s delightful monograph \cite{littlewood1968some}) which spurred intensive research throughout the past seventy years. Here, we mention recent breakthrough results solving Littlewood's conjecture on the existence of flat polynomials \cite{flatexist,beck,BBflat,Kahane} (see also \cite{Borweinexcursion,borweinBarker,Erdospb,Mont-Litt,newman,odlyzko} and the references therein for a historical account of such problems and the link with Barker sequences and several related engineering problems).\\ Among various classes of Littlewood polynomials two specific examples
attracted particular attention: the Fekete and the Rudin-Shapiro polynomials \cite{ErdMahler1,Erdunitcircle,jedwab2013littlewood,Rodgers,saffariRudin}.
For a given prime $p,$ we recall that the corresponding Fekete polynomial is defined by
$$F_p(z):=\sum_{n=1}^{p-1}\left(\frac{n}{p}\right)z^n,$$
where $\left(\frac{\cdot}{p}\right)$ is the Legendre symbol modulo $p$.
The family of Fekete polynomials appears frequently in the context of both number theoretic and analytic problems and has been extensively studied for over a century, see \cite{BaMo}, \cite{BPW}, \cite{BSh}, \cite{BC}, \cite{borwein2001extremal}, \cite{BC-merit}, \cite{Chow}, \cite{FeketePolya}, \cite{Heilbr}, \cite{Hoholdtmerit}, \cite{JKS}, \cite{jedwab2013littlewood}, \cite{JHH}, \cite{Mont-large} for some of the results.
Dirichlet was perhaps the first to discover the identity
\[L\left(s,\left(\frac{\cdot}{p}\right)\right)\Gamma(s)=\int_0^1 \frac{(-\log x)^{s-1}}{x}\frac{F_p(x)}{1-x^p}dx,\] where $L(s,\chi)$ is the Dirichlet $L$-function associated to a Dirichlet character $\chi$ modulo $p$.
 Consequently, if $F_p$ has no zeros for $0<x<1,$ then $L\left(s,\left(\frac{\cdot}{p}\right)\right)>0$ for all $0<s<1,$ which in turn refute  the existence of a putative Siegel zero. The positivity of $L(s,\chi_d)$ however (where $\chi_d$ is the Kronecker symbol associated to a fundamental discriminant $d$), is only known for a positive proportion of fundamental discriminants $d$ by the work of Conrey and Soundararajan  \cite{ConreySound}.

\subsection{The Mahler measure of Fekete polynomials} A classical quantity related to the distribution of complex zeros of an algebraic polynomial $P\in\mathbb{C}[x]$ is the Mahler measure of $P$ given by
\[M_0(P):=\exp\left(\int_{0}^1\log |P(e(t))|dt\right),\]
where $e(t):= e^{2\pi i t}$.
The problem of determining  an asymptotic formula for the Mahler measure of Fekete polynomials has attracted considerable attention.
Littlewood \cite{Littlewood-Mahler} was the first to prove a general upper bound for the Mahler measure of Littlewood polynomials, from which one has
$M_0(F_p)\le (1-\varepsilon_0)\sqrt{p-1}$ for some small fixed $\varepsilon_0>0.$
Using subharmonic methods, Erd\'elyi and Lubinskii~\cite{LE} established the lower bound $M_0(F_p)\ge (\frac{1}{2}-\varepsilon)\sqrt{p},$ which has since been improved in~\cite{Lower-Mah} to $M_0(F_p)\ge (\frac{1}{2}+c_1)\sqrt{p},$ for some small value of $c_1>0.$
In \cite{ErdMahler1} (see also recent survey~\cite{ErdSurv}), Erd\'elyi writes the following about the problem of finding an asymptotic formula for $M_0(F_p)$:\\
\begin{center}
``{\it \dots this problem seems to be beyond reach at the moment. Not even a (published or unpublished)
conjecture about the asymptotic seems to be known."}\\
\end{center}
Our first result resolves this problem.
\begin{thm}\label{cor: Mahnorms}
Let ${G}_{\mathbb{X}}$ be the random process on $C[0,1]$ defined by
\begin{equation}\label{process}
G_{\X}(t):= \sum_{m\in \mathbb{Z}} \frac{e(t)-1}{2\pi i (m-t)} \X(m), \hspace{2mm}t\in[0,1]
\end{equation}
where $\{\X(m)\}_{m\in \mathbb{Z}}$ are I. I. D. random variables taking the values $\pm 1$ with equal probability $1/2$.  Then we have
\[\lim_{p\to \infty}\frac{M_0(F_p)}{\sqrt{p}}= k_0:=\int_{0}^1\mathbb{E}\left( \log |{G}_{\mathbb{X}}(t)|\right)dt,\]
where the constant $k_0=0.74083\dots$ is effectively computable (see Remark \ref{rem:cst} for details).
\end{thm}
\begin{rem}\label{rem:mahler}
     There are a series of related results concerning the average behaviour of the Mahler measure for a typical Littlewood polynomial (with coefficients $\pm 1$). Here, a general result due to Choi and Erdélyi~\cite{Aveg-Mah}, relying on a deep work of Konyagin and Schlag~\cite{Kon-Sch}, states that on average
\begin{equation}\label{Mahlaverage} \lim_{n\to\infty}\frac{1}{2^{n+1}}\sum_{f\in \mathcal{L}_n}\frac{M_0(f)}{\sqrt{n}}=e^{-\gamma/2}=0.749306\dots, \end{equation} where
$\mathcal{L}_n$ denotes the set of polynomials of degree $n$ with $\pm 1$ coefficients. Our constant $k_0$ differs from the one in \eqref{Mahlaverage} and from the previous examples in the literature. On a related note, relying on a beautiful result of Rodgers \cite{Rodgers},
Erdélyi \cite{ErdMahler1} recently proved Saffari's conjecture about the  asymptotic behavior for the Mahler measure of the Rudin-Shapiro polynomials $P_k(z)$ (of degree $2^k-1$), which states that 
$$ \lim_{k\to \infty} \frac{M_0(P_k)}{\sqrt{2^{k+1}}}= e^{-1/2}.$$
%Since 
%\[M_0(P)=\lim_{q\to 0}\|P\|^q_q\]
%the problem of determining Mahler measure is closely related to another well-studied question of estimating $L_q$ norms of the constrained polynomials $P\in\mathcal{L}_n.$
\end{rem}
\subsection{$L_q$ norms of Fekete polynomials} A well-known result of Borwein and Lockhart~\cite{Bo-Lo} states that for every $0<q<\infty,$
\[\lim_{n\to\infty}\frac{1}{2^{n+1}}\sum_{f\in \mathcal{L}_n}\frac{M_q(f)}{\sqrt{n}}=\Gamma\left(1+\frac{q}{2}\right)^{1/q},\]
where 
$$M_q(P):=\left(\int_0^1|P(e(t))|^q dt\right)^{1/q}.$$
In the case of Fekete polynomials, the analogous problem has been extensively studied.
%A beautiful result of Montgomery~\cite{Mont-large} shows that 
%\[\sqrt{p}\log \log p\ll \max_{|z|=1}|F_p(z)|\ll \sqrt{p}\log p,\]
%but the exact order of the growth of is still unknown.
Interestingly, Borwein and Choi~\cite{BC} found the exact expressions for $\vert|F_p\vert|_4^4$ in terms
of the class number of $\mathbb{Q}[\sqrt{-p}].$ Erd\'elyi~\cite{E-Lpgrowth} determined the exact order of the growth ($\asymp\sqrt{p}$) of $M_q(F_p)$ for every $q>0.$
These investigations culminated in the work of G\"unther and Schmidt~\cite{GS-Lp}, who showed that for every integer $k\geq 1$ we have
$$\lim_{p\to \infty}\frac{M_{2k}(F_p)}{\sqrt{p}}= (\ell_{2k})^{1/2k},$$ where the constant $\ell_{2k}$ is given via a series of rather complicated recursive combinatorial identities (Theorem 2.1 in~\cite{GS-Lp}). Our next corollary recovers their result and extends it to all values $q>0$ via a different method.
\begin{cor}\label{cor:lqnorms}
Let ${G}_{\mathbb{X}}(t)$ be the random process defined in \eqref{process}. Then for any $0<q<\infty,$ we have
\[\lim_{p\to\infty}\frac{M_q(F_p)}{\sqrt{p}}=k_q:=\left(\int_{0}^1\mathbb{E} (|G_{\mathbb{X}}(t)|^q) dt\right)^{1/q}.\]
\end{cor}
\begin{rem}In the case $q=\infty,$ a beautiful result of Montgomery~\cite{Mont-large} shows that 
\[\sqrt{p}\log \log p\ll M_{\infty}(F_p)=\max_{|z|=1}|F_p(z)|\ll \sqrt{p}\log p,\]
but the exact order of the growth is still unknown.\\
\end{rem}
\subsection{Distributional result.}
Establishing appealing conjectures of Montgomery and Saffari, Rodgers~\cite{Rodgers} recently proved a distributional result for the appropriately normalized family of the Rudin-Shapiro polynomials $P_k(z)$, showing that one has convergence to the uniform distribution on the unit disc $\{|z|\leq 1\}$ in this case. More precisely, one has
\begin{equation}\label{RS-gaussian}
\lim_{k\to \infty}\text{m} \left(t\in [0,1]:\frac{P_k(e(t))}{\sqrt{2^{k+1}}}\in U\right)= \frac{1}{\pi} |U|,
\end{equation}
for any rectangle $U\subset\mathbb{C}$, where $\text{m}$ is the standard Lebesgue measure, and $|U|$ is the area of $U$. Subsequently, Erdélyi \cite{Erdunitcircle} used \eqref{RS-gaussian} to show that a negligible proportion of the zeros of Rudin-Shapiro polynomials lie on the unit circle. This is in a strong contrast with a rather surprising result of Conrey, Granville, Poonen and Soundararajan~\cite{CGPS}, who showed that there exists a constant $1/2<c<1$
such that
$$\{z : |z| = 1\ and\ F_p(z) = 0\} \sim c p,$$
with $0.500668  < c < 0.500813 .$ %\footnote{Using the zero detecting software one can check that for all primes $ p<1000,$ roughly half of the zeros of $F_p$ lie on the unit circle.} \\ %As in the random case~\cite{HN}, complex zeros tend to equidistribute around the unit disc in a quantitative way as $p\to\infty,$ see~\cite{CGPS} for the details.\\
Our next theorem provides an analogue of the Montgomery-Saffari conjecture for Fekete polynomials.
\begin{thm}\label{main_convprocess}
Let $U\subset\mathbb{C}$ be a rectangle with sides parallel to the coordinate axes. Then we have
$$\lim_{p\rightarrow \infty} \text{m} \left(t\in [0,1]:\frac{F_p(e(t))}{F_p(\xi_p)}\in U\right)=\mathbb{P}(G_{\X}(\theta) \in U),
$$
where $\theta$ is a random variable uniformly distributed on $[0,1]$.
%\begin{align*} \lim_{p\rightarrow \infty} \text{m} \left(\theta\in [0,1]:\frac{F_p(e^{i\theta})}{F_p(\xi_p)}\in U\right) & =
%\ex \left\{ \mathbb{P}(G_{\X}(t) \in U) \right\} = \ex \left\{\text{m}  \left(t\in [0,1]: G_{\mathbb{X}}(t)\in U\right)\right\} \\ 
%& =\ex \left( \int_{0}^{1} \mathbbm{1}_{U}(G_{\mathbb{X}}(t)) dt\right).\end{align*}}
 \end{thm} We refer our interested reader to the plot above for the visualisation of the random process $G_{\X}.$

\begin{figure}
\centering
\includegraphics[scale=0.5]{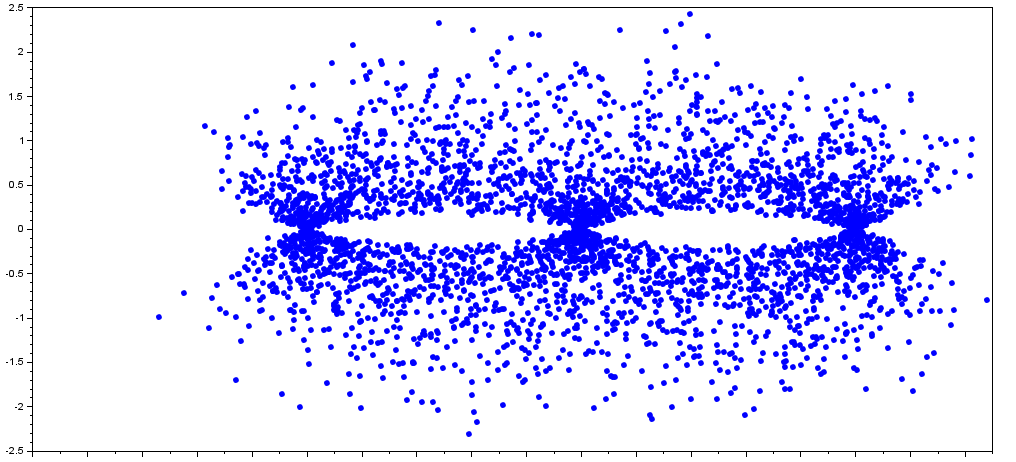}
\caption{Realization of $G_{\X}(t), t\in (0,1)$} 
\label{Random Fekete}
\end{figure}
We now describe yet another relation of our results to the classical problem on the distribution of random trigonometric polynomials.
%An important problem is to understand the distribution and the size of trigonometric polynomials $\sum_{k\le n}a_ke(k\theta)$ for uniformly distributed $\theta\in[0,1]$ where $a_k$ are interesting deterministic or randomly selected coefficients. A prototypical example
A seminal result of Salem and Zygmund~\cite{S-Z}, shows that 
if the coefficients $a_k$ are independent identically distributed random variables taking values $\pm 1$ with probability $1/2$ each,
then almost surely
\[\frac{1}{\sqrt{N}}\sum_{k\le N}a_ke(k\theta) \xrightarrow{a.s.} \mathcal{CN}(0,1),\]
where the right hand side denotes a standard complex Gaussian distribution.
In the case of sequences $a_n$ possessing some multiplicative properties, the situation is more subtle and often leads to various arithmetic consequences.  
See for example the work of Hughes
and Rudnick~\cite{Rud-Hugh} on lattice points in annuli. Another example is a recent result of Benatar, Nishry and Rodgers~\cite{B-N-R}, which shows convergence to the Gaussian distribution when $a_n$ are random (Rademacher or Steinhaus) multiplicative functions.\\
Upon applying Fourier inversion, the distribution of a normalized trigonometric polynomial 
\begin{equation}\label{harpchar}
\frac{1}{\sqrt{H}}\sum_{k\le H}\left(\frac{k}{p}\right) e(k\theta),
\end{equation}
for various ranges of $H$ is essentially equivalent to the statistical behaviour of short character sums
$$\displaystyle{\frac{1}{\sqrt{H}}\sum_{x\le n\le x+H}\left(\frac{n}{p}\right)},$$
for a randomly chosen interval $[x,x+H]$ with $x\in \{0, 1, \dots, p-1\}$.
It follows from a recent work of Harper \cite{Harp-moving-character} that for almost all primes $p$ the sum \eqref{harpchar}
converges to the standard Gaussian in the range $H\sim (\log p)^A$ (where $A\geq 1$ is a fixed constant) when $\theta\in [0,1]$ is selected uniformly at random. In contrast, Harper \cite{Harp-moving-character} shows the existence of a sequence of primes $p$ with $p\to\infty,$ where the corresponding sums in \eqref{harpchar} do not satisfy the central limit theorem for this range of $H$.\\ It remains a deep mystery to understand the behaviour of \eqref{harpchar} for different ranges of $1\le H\le p$ and our Theorem \ref{main_convprocess} shows that in the extreme case $H=p,$ the distribution of \eqref{harpchar}  is governed by the random process given by \eqref{process}, and is in particular non-Gaussian.\\

%%%%%%%%%%%%%%%%%%%%%%%%%%%%%%%%%%%%%%%%%%%%%%%%%%%%%%

\section{Plan of attack}
We now briefly outline our general strategy.
We consider the auxiliary function
 \[ G_p(k,t):=  \frac{F_p\left(e\left(\frac{k+t}{p}\right)\right)}{F_p(\zeta_p)},\] 
 where $\zeta_p:= e(1/p)$. Note that 
 \begin{equation}\label{eq: GaussSum}
    \begin{aligned}
        F_p(\zeta_p)= \sum_{n = 1}^p \left(\frac{n}{p}\right) e(n/p) = \begin{cases}
             \sqrt{p}, & \text{ if } p \equiv 1 \pmod 4,  \\
             i \sqrt{p}, & \text{ if } p \equiv 3 \pmod 4.
        \end{cases}
    \end{aligned}
    \end{equation} is the Gauss sum associated to the Legendre symbol modulo $p$.
 
 The key starting point in our investigations is the identity borrowed from \cite[Formulas $(2.2)$ and $(2.3)$]{CGPS} which implies that
\[
  G_p(k, t)= \sum_{|m|\leq (p-1)/2} \bigg(\frac{k-m}p\bigg)  \frac{e(t)-1} {p(e( \frac{m-t}p)-1) }.
\]
%then we have 
%$$G_p(k,t)=  \frac{F_p\left(e\left(\frac{k+t}{p}\right)\right)}{F_p(\zeta_p)}.$$
Given $|m|\le (p-1)/2,$ we now aim to exploit the ``randomness" coming from the action of the shifts $\bigg(\frac{k-m}p\bigg)$ for various $0\le k\le p-1.$ 
To this end, we view $(G_p(k, t))_{t\in [0,1]}$ as a random process on the finite probability space $\mathbb{F}_p$ equipped with the uniform measure. 
%Let $\{\X(m)\}_{m\in \mathbb{Z}}$ be i.i.d. random variables taking the values $\pm 1$ with equal probability $1/2$, and consider the random process $(G_{\X}(t))_{t\in[0,1]}$ defined by 
%$$ G_{\X}(t)= \sum_{m\in \mathbb{Z}} \frac{e(t)-1}{2\pi i (m-t)} \X(m).$$ 
Our goal is to prove the following result.
\begin{thm}\label{convprocess}
For any bounded continuous function $\varphi : C[0,1]\to \mathbb{C}$ we have 
\begin{equation}\label{ConvLaw}
\lim_{p\to \infty} \frac 1p \sum_{k=0}^{p-1} \varphi\left(\frac{F_p\left(e\left(\frac{k+t}{p}\right)\right)}{F_p(\zeta_p)}\right)= \ex\varphi (G_{\X}),\end{equation}
where $G_{\X}$ is defined in \eqref{process}.
 \end{thm}
In probabilistic terms, Theorem \ref{convprocess} is equivalent to the statement that the sequence of random processes $(G_p(k, t))_{t\in [0,1]}$ converges in law (in the space of continuous functions on $[0,1]$) to the process  $(G_{\X}(t))_{t\in [0, 1]}$. 
%In particular, we aim to show that for any continuous function $\varphi : C[0,1]\to \mathbb{C}$ we have 
%\begin{equation}\label{ConvLaw}
%\lim_{p\to \infty} \frac 1p \sum_{k=0}^{p-1} \varphi (G_{p,k}(t))= \ex\varphi (G_{\X}).
%\end{equation} 
There are two main ingredients to prove this convergence in law. The first is proving the convergence of finite distributions, that is the convergence of finite moments, which will be established in Section \ref{finitedib}. The second ingredient is to show that the  sequence of random processes $(G_p(k, t))_{t\in [0,1]}$ is relatively compact, which is equivalent, by Prokhorov's theorem, to verifying the ``tightness'' of these processes (this is the content of Section \ref{tightness}).\\

Now we can write the logarithmic Mahler measure of $F_p$ as follows
\[
\log M_0(F_p)=\int_0^1\log  | F_p(e(t))| dt = \frac 1p \sum_{k=0}^{p-1}    \int_0^1 \log \bigg| F_p\bigg( e\bigg( \frac{k+t}p\bigg) \bigg) \bigg|  dt
\]
\[
= \log \sqrt{p} + \frac 1p \sum_{k=0}^{p-1}    \int_0^1 \log  |G_p(k,t)|   dt, 
\]
since $|F_p(\zeta_p)|=\sqrt{p}$.
We note that the functional $\ell(\phi)=\int_{0}^1\log |\phi(t)|dt$ is not continuous on $C[0,1]$ and so we cannot apply Theorem \ref{convprocess} directly.
We thus have to deal with the issues of uniform ``log-integrability," which is a classical topic in the theory of random and deterministic Fourier series and our treatment is inspired by the work of Nazarov, Nishry and Sodin \cite{NNS}.
Roughly speaking, to overcome this obstacle, we instead consider the regularized functional $\ell_{\epsilon}(\phi)=\int_{\varepsilon}^{1-\varepsilon}\log (|\phi(t)|)\mathbbm{1}_{|\phi(t)|\ge \varepsilon}dt,$ to which our Theorem \ref{convprocess} applies.
We then show that uniformly for each function in our family, {\it either} $|\frac{G_{p,k}'(t)}{e(t)-1}|\gg 1,$ or $|\frac{G_{p,k}''(t)}{e(t)-1}|\gg 1,$ which then allows us, after some technical work, to adequately control the logarithmic singularities on both random and  deterministic sides.

%%%%%%%%%%%%%%%%%%%%%%%%%%%%%%%%%%%%%%%%%%%%%%%%%%%%%%%%%%%%%%%%%%%%%%%%%%%%%%%%%%%%%%%%%%%%%%%%%%%%%%%%%%%%%%%%%%%%%%%%%%%%%%%%%%%%%%%%%%%%%%%%%%%

\section{Preparatory results}
We begin with several simple observations concerning the random process.

\begin{lem}\label{lem:smooth-process}
The functions  $t \rightarrow \Re G_{\X}(t)$ and $t \rightarrow \Im G_{\X}(t)$ are almost surely real-analytic on $[0,1]$. 
\end{lem}
\begin{proof}
    We only prove the result for $\Re
    G_{\X}(t)$ since the proof for $\Im G_{\X}(t)$ is similar. We write 
    $$ \Re G_{\X}(t)= \Re \sum_{m\in \mathbb{Z}} \frac{e(t)-1}{2\pi i (m-t)} \X(m)=\sum_{|m|\ge 2}\frac{\sin(2\pi t)}{2\pi(m-t)} \X(m)+\sum_{m=-1,0,1}\frac{\sin(2\pi t)}{2\pi (m-t)} \X(m),$$
    and observe that the second sum defines an analytic function over all realizations of $\mathbb{X}(-1),\mathbb{X}(0)$ and $\mathbb{X}(1).$ Since $\sin(2\pi t)$ is analytic it is enough to show that the random function
    $$ G_{\X,2}(t)= \sum_{|m|\ge 2} \frac{\X(m)}{ m-t}$$
    is almost surely analytic. We notice that 
    $$ G'_{\X,2}(t)= \sum_{|m|\ge 2} \frac{\X(m)}{ (m-t)^2}$$
    is absolutely convergent over all realizations of $\X.$ Similarly, we almost surely have the bound $|G_{\X,2}^{(k)}(t)|\ll \sum_{m\ge 1}\frac{1}{m^k}\ll 1,$ which holds uniformly for all $t\in [0,1].$ This implies the (real) analyticity of $G_{\X}(t).$
\end{proof}
\begin{lem}\label{approx-distribution}
Let $U \subset \mathbb{C}$ be a rectangle with sides parallel to the coordinate axes.  Then, almost surely,
$\textup{m} \left(t\in [0,1]: G_{\X}(t)\in \partial U\right)=0.$
\end{lem}

\begin{proof}
We note that it is enough to show that for any $a\in\mathbb{R}$ we have almost surely $$\text{m} (t\in [0,1]: \Re G_{\X}(t)=a)=0 \textup{ and } \text{m} (t\in [0,1]: \Im G_{\X}(t)=a)=0.$$
By the previous lemma, the functions $\Re G_{\X}(t)$ and $\Im G_{\X}(t)$ are almost surely non-constant real analytic functions and so the claim follows immediately.

\end{proof}

%%%%%%%%%%%%%%%%%%%%%%%%%%%%%%%%%%%%%%%%%%%%%%%%%%%%%%%%%%%%%%%%%%%%%%%%%%%%%%%%%%%%%%%%%%%%%%%%%%%%%%%%%%%%%%%%%%%%%%%%%%%%%%%%%%%%%%%
\section{Convergence to the random process: proof of Theorem \ref{convprocess} }\label{ProcessFekete}
In order to prove Theorem \ref{convprocess}, we introduce truncations 
$$G_{\X, p}(t)= \sum_{|m|\leq (p-1)/2} \frac{e(t)-1}{2\pi i (m-t)} \X(m)$$
and further define  
$$ \widetilde{G}_{\X,p}(t)= \sum_{|m|\leq (p-1)/2} \frac{e(t)-1} {p(e( \frac{m-t}p)-1)} \X(m).$$ To simplify the notation we let  
\begin{equation}\label{def-alpha} \alpha_p(m; t):= \frac{e(t)-1} {p(e( \frac{m-t}p)-1)} \end{equation} and note that 
 for all $|m|\leq (p-1)/2$ and $t\in [0, 1]$ we have uniformly
\begin{equation}\label{ApproxCoefProcess}
\alpha_p(m; t)= \frac{e(t)-1}{2\pi i (m-t)}+O\left(\frac{1}{p}\right).
\end{equation} 
\subsection{Convergence in the sense of finite distributions }\label{finitedib}
In this section, we will prove the following proposition.
%Our starting point is the following moment computation.
\begin{pro}\label{ProMomentsFiniteDistribution}
Let $J\geq 1$ be an integer, $0 \leq t_1 < t_2 < \cdots < t_J \leq 1$ be real numbers and $r_1,\dots, r_J$ and $s_1,\dots,s_J$ be non-negative integers. Then we have 
$$ \frac{1}{p} \sum_{k \in \mathbb{F}_p} \prod_{j=1}^J G_p(k, t_j)^{r_j} \overline{G_p(k, t_j)}^{s_j} = 
\ex\left(\prod_{j=1}^J 
G_{\X}(t_j)^{r_j} \overline{G_{\X}(t_j)}^{s_j} \right) + O_{s, r}\left(\frac{(\log p)^{r+s}}{\sqrt{p}}\right),
$$
where $s= s_1+\cdots+s_J$ and $r=r_1+\cdots + r_J$.
\end{pro}
To prove this result we first derive a natural approximation property.
\begin{lem}\label{SecondMomentApprox}
Let $p$ be a large prime. Then uniformly for $t\in [0,1]$ we have
$$\ex\left(\left| G_{\X}(t)-\widetilde{G}_{\X,p}(t)\right|^2\right)\ll \frac{1}{p}, $$
where the implicit constant is absolute.
\end{lem}

\begin{proof} By Minkowski's inequality we have 
\begin{equation}\label{MinkoMoment}
\ex\left(\left| G_{\X}(t)-\widetilde{G}_{\X,p}(t)\right|^2\right)\ll \ex\left(\left| G_{\X}(t)-G_{\X,p}(t)\right|^2\right)  + \ex\left(\left| G_{\X,p}(t) - \widetilde{G}_{\X,p}(t)\right|^2\right).
\end{equation}
We shall bound each moment separately. Since the $\X(m)$ are independent with mean $0$ and variance $1$, we have uniformly for $t\in [0, 1]$
\begin{equation}\label{SecondMomentApprox1}
\ex\left(\left| G_{\X}(t)-G_{\X,p}(t)\right|^2\right) = 
\sum_{|m|>(p-1)/2} \frac{|e(t)-1|^2}{4\pi^2 (m-t)^2} \ll \sum_{|m|>(p-1)/2} \frac{1}{m^2} \ll\frac{1}{p}.
\end{equation}
%since for all $t\in [0, 1]$ we have %$|e(t)-1|\leq 2$  and 
%$$\sum_{|m|>(p-1)/2} \frac{1}{(m-t)^2} \leq \sum_{|m|>(p-1)/2} \frac{1}{(m-1)^2} \ll  \frac{1}{p}.$$ 
Furthermore, by \eqref{ApproxCoefProcess} we derive 
\begin{equation}\label{SecondMomentApprox2}
\ex\left(\left| G_{\X,p}(t) - \widetilde{G}_{\X,p}(t)\right|^2\right)= 
\sum_{|m|\leq \frac{p-1}{2}} |e(t)-1|^2\left| \alpha_p(m; t)- \frac{e(t)-1}{2\pi i (m-t)}\right|^2 \ll \frac{1}{p},
\end{equation}
uniformly for $t\in [0,1]$. Inserting the estimates \eqref{SecondMomentApprox1} and \eqref{SecondMomentApprox2} in \eqref{MinkoMoment} completes the proof.
\end{proof}

\begin{proof}[Proof of Proposition \ref{ProMomentsFiniteDistribution}] 
%To simplify the notation we let  
%$$ \alpha_p(m; t):= \frac{e(t)-1} {p(e( \frac{m-t}p)-1)}.$$
We write  
\begin{align*}
 &\prod_{j=1}^J G_p(k, t_j)^{r_j} \overline{G_p(k, t_j)}^{s_j}\\
  = & \prod_{j=1}^J \left(\sum_{|m|\leq (p-1)/2} \bigg(\frac{k-m}p\bigg) \alpha_p(m ;t_j)\right)^{r_j}\left(\sum_{|m|\leq (p-1)/2} \bigg(\frac{k-m}p\bigg) \overline{\alpha_p(m ;t_j)}\right)^{s_j}\\
 = & \sum_{\mathbf{m_1}, \dots, \mathbf{m_J}} \prod_{j=1}^J\prod_{\ell=1}^{r_j+s_j} \left(\frac{k-m_{j, \ell}}{p}\right)\prod_{j=1}^J \alpha_p(\mathbf{m_j}; t_j),\\
\end{align*}
where the quantity
$$ \mathbf{m_j}= (m_{j,1}, \dots, m_{j, r_j}, m_{j, r_j+1}, \dots, m_{j, r_j+s_j})$$
ranges over all $(r_j+s_j)$-tuples of integers $|m_{j, l}|\leq (p-1)/2$, and 
$$  \alpha_p(\mathbf{m_j}; t_j)= \alpha_p(m_{j,1}; t_j)\cdots \alpha_p(m_{j,r_j}; t_j) \overline{\alpha_p(m_{j,r_j+1}; t_j)\cdots \alpha_p(m_{j,r_j+s_j}; t_j)}. $$
Therefore, we obtain 
$$  \frac{1}{p} \sum_{k \in \mathbb{F}_p}\prod_{j=1}^J G_p(k, t_j)^{r_j} \overline{G_p(k, t_j)}^{s_j}= \sum_{\mathbf{m_1}, \dots, \mathbf{m_J}} \prod_{j=1}^J \alpha_p(\mathbf{m_j}; t_j)\frac{1}{p} \sum_{k \in \mathbb{F}_p}\prod_{j=1}^J\prod_{\ell=1}^{r_j+s_j} \left(\frac{k-m_{j, \ell}}{p}\right).$$
We now use the Weil bound for character sums (see \cite[Corollary 11.24]{IwKo}), which is a consequence of Weil's proof of the Riemann Hypothesis for curves over finite fields, to get
$$ 
\frac{1}{p} \sum_{k \in \mathbb{F}_p}\prod_{j=1}^J\prod_{\ell=1}^{r_j+s_j} \left(\frac{k-m_{j, \ell}}{p}\right) = \ex\left(\prod_{j=1}^J\prod_{\ell=1}^{r_j+s_j} \X(m_{j, \ell})\right) + O\left(\frac{r+s}{\sqrt{p}}\right).
$$
Thus, we derive 
\begin{align}\label{MomentsWeil}
\frac{1}{p} \sum_{k \in \mathbb{F}_p}\prod_{j=1}^J G_p(k, t_j)^{r_j} \overline{G_p(k, t_j)}^{s_j} & = \sum_{\mathbf{m_1}, \dots, \mathbf{m_J}} \prod_{j=1}^J \alpha_p(\mathbf{m_j}; t_j) \ex\left(\prod_{j=1}^J\prod_{\ell=1}^{r_j+s_j} \X(m_{j, \ell})\right)  + E_1 \nonumber \\
& = \ex\left(\sum_{\mathbf{m_1}, \dots, \mathbf{m_J}} \prod_{j=1}^J \alpha_p(\mathbf{m_j}; t_j) \prod_{j=1}^J\prod_{\ell=1}^{r_j+s_j} \X(m_{j, \ell})\right) +E_1 \nonumber\\
& = \ex\left(\prod_{j=1}^J 
\widetilde{G}_{\X, p}(t_j)^{r_j} \overline{\widetilde{G}_{\X, p}(t_j)}^{s_j}\right)+ E_1,
\end{align}

where 
\begin{align*}
 E_1 &\ll_{r, s} \frac{1}{\sqrt{p}} \sum_{\mathbf{m_1}, \dots, \mathbf{m_J}} \prod_{j=1}^J |\alpha_p(\mathbf{m_j}; t_j)| = \frac{1}{\sqrt{p}}\prod_{j=1}^J \left(\sum_{|m|\leq (p-1)/2} |\alpha_p(m ,t_j)|\right)^{r_j+s_j}  \ll_{r, s} \frac{(\log p)^{r+s}}{\sqrt{p}},
 \end{align*}
since uniformly for $t\in [0,1]$ we have $|\alpha_p(m ,t)|\ll 1/|m|$ for $|m|\geq 2$ and hence
$$ \sum_{|m|\leq (p-1)/2} |\alpha_p(m ,t)| \ll 1 + \sum_{2\leq |m|\leq (p-1)/2} \frac{1}{|m|} \ll \log p.$$
Let $\Y_j= G_{\X}(t_j)- \widetilde{G}_{\X, p}(t_j).$ Then, it follows from Lemma \ref{SecondMomentApprox} that there exists an absolute constant $C>0$ such that for all $1\leq j\leq J$ we have $ \ex(|\Y_j|^2)\leq C/p$. Therefore, by the Cauchy-Schwarz inequality we obtain 
\begin{equation}\label{CauchyMomentsDiff}
\ex\left(|\Y_{j_1}\Y_{j_2}\cdots \Y_{j_{\ell}}|\right) \leq \ex\left(|\Y_{j_1}|^2\right)^{1/2} \ex\left(|\Y_{j_2}\cdots \Y_{j_{\ell}}|^2\right)^{1/2}\ll_{\ell} \frac{1}{\sqrt{p}},
\end{equation}
for any $\ell\geq 1$ and $1\leq j_1\leq j_2 \leq \dots j_{\ell}\leq J.$ Hence, using that 
$\sup_{t\in[0,1]}|\widetilde{G}_{\X, p}(t)|\ll \log p$ we deduce that 
\begin{align*}
& \ex\left(\prod_{j=1}^J 
G_{\X}(t_j)^{r_j} \overline{G_{\X}(t_j)}^{s_j} \right)\\
= & \ex\left(\prod_{j=1}^J 
(\widetilde{G}_{\X, p}(t_j)+\Y_j)^{r_j} \big(\overline{\widetilde{G}_{\X, p}(t_j)}+\overline{\Y_j}\big)^{s_j} \right)\\
= & \ex\left(\prod_{j=1}^J\left(\sum_{\ell=0}^{r_j} \binom{r_j}{\ell}  \Y_j^{\ell} \widetilde{G}_{\X, p}(t_j)^{r_j-\ell}\right)
\left(\sum_{m=0}^{s_j} \binom{s_j}{m} \overline{\Y_j}^{m}  \overline{\widetilde{G}_{\X, p}(t_j)}^{s_j-m}\right)\right)\\
= & \sum_{\substack{(\ell_1, \dots, \ell_{2J}) \\ 0\leq \ell_j\leq r_j \\ 0\leq \ell_{j+J} \leq s_j}} 
\ex\left(\prod_{j=1}^J \binom{r_j}{\ell_j} \binom{s_j}{\ell_{j+J}} \Y_j^{\ell_j}\overline{\Y_j}^{\ell_{j+J}} \widetilde{G}_{\X, p}(t_j)^{r_j-\ell_j}\overline{\widetilde{G}_{\X, p}(t_j)}^{s_j-\ell_{j+J}}\right) \\
= & \ex\left(\prod_{j=1}^J\widetilde{G}_{\X, p}(t_j)^{r_j}\overline{\widetilde{G}_{\X, p}(t_j)}^{s_j}\right) + 
O_{s, r}\Bigg((\log p)^{s+r} \sum_{\substack{(\ell_1, \dots, \ell_{2J})\neq (0, 0, \dots, 0) \\ 0\leq \ell_j\leq r_j \\ 0\leq \ell_{j+J} \leq s_j}} 
\ex\left(\prod_{j=1}^J |\Y_j|^{\ell_j+\ell_{j+J}}\right)\Bigg)\\
= & \ex\left(\prod_{j=1}^J\widetilde{G}_{\X, p}(t_j)^{r_j}\overline{\widetilde{G}_{\X, p}(t_j)}^{s_j}\right)  + O_{s, r}\left(\frac{(\log p)^{s+r}}{\sqrt{p}}\right),
\end{align*}
by \eqref{CauchyMomentsDiff}. Combining this with \eqref{MomentsWeil} completes the proof.

\end{proof}
 \subsection{Tightness of the process}\label{tightness}
 
We now invoke Prokhorov's theorem \cite[Theorem $6.1$]{Bill} which asserts that if a sequence of probability measures is tight, then it must be relatively compact. In this section, we will prove the tightness of the process $G_p(.,t)$ using the following criterion due to Kolmogorov   (see \cite[Proposition B.11.10]{KoBook}). %or \cite{Centsov}).
 
 \begin{lem}\label{Kolmogorov}
 Let $(L_p(t))_{t \in  \left[0,1\right]}$ be a sequence of $C(\left[0,1\right])$- valued processes. % such that  $(L_p(0))_{p}$ is tight.
 If there exist constants $\alpha >0, \delta>0$ and $C \geq 0$ such that for any $p$ and any $s<t$ in $\left[0,1\right]$ we have 
 $$ \mathbb{E}(\left\vert L_p(t)-L_p(s)\right\vert^{\alpha})\leq C\left\vert t-s\right\vert^{1+\delta},   $$
 then the sequence $(L_p(t))_{t \in  \left[0,1\right]}$ is tight. \end{lem}
To set things up, we begin by recalling several classical facts needed in the proof. The first one is due to Gauss.
 \begin{lem}\label{quadsumLegendre}
 Let $p\geq 3$ be a prime number. For all integers $n$ we have
 
$$   \sum_{k=1}^{p} \left(\frac{k(k+n)}{p}\right)= \begin{cases}
p-1 &\textrm{ if } p\mid n,\\
-1 & \textrm{ otherwise.}
\end{cases} $$
 \end{lem} 
 The next lemma is the well-known inequality due to Bernstein \cite{Bernstein}.
 \begin{lem}\label{Bernstein}
For any trigonometric polynomial $P$ of degree $n$, we have 
$$ \max_{\vert z\vert=1} \vert P'(z)\vert \leq n \max_{\vert z\vert=1} \vert P(z)\vert  .$$
\end{lem} 
Montgomery \cite{mont} proved that for any $p>2$ we have
$ \max_{\vert z\vert=1} \vert F_p(z)\vert \ll \sqrt{p}\log p,$ which together with Lemma \ref{Bernstein} yields 

\begin{equation}\label{boundsupderivative} \max_{\vert z\vert=1} \vert F'_p(z)\vert   \ll p^{3/2} \log p. \end{equation}
 Our goal now is to prove the following equicontinuity result.
 \begin{lem}\label{lemlarge}
Let $p$ be an odd prime. There exists an absolute constant $C,$ independent of $p,$ such that for all $0\leq s<t\leq 1$, 
$$\mathbb{E}\left(\left\vert G_p(k,t)-G_p(k,s) \right\vert^{2}\right) := \frac{1}{p}\sum_{k=0}^{p-1} \left\vert G_p(k,t)-G_p(k,s) \right\vert^{2} \leq C |t-s|^{3/2}.$$
\end{lem} 
\begin{proof}
% For $0\leq k \leq p-1,$ we recall that $G_p(k,0)=\left(\frac{k}{p}\right)$ (and $F_p(1)=0$). Furthermore, 
Applying the mean-value theorem together with \eqref{boundsupderivative} yields
\begin{align*}\label{boundsmall} \vert G_p(k,t)-G_p(k,s) \vert &= \frac{1}{\sqrt{p}}\left\vert F_p\left(e\left(\frac{k+t}{p}\right)\right)-F_p\left(e\left(\frac{k+s}{p}\right)\right)\right\vert \\
& \ll  p(\log p) \left\vert \frac{k+t}{p}-\frac{k+s}{p} \right\vert \\
& \ll   |t-s| (\log p).
\end{align*} Consequently, if $\vert t-s\vert \leq \frac{1}{\log^4 p},$ we have
$$\mathbb{E}\left\vert G_p(k,t)-G_p(k,s) \right\vert^{2} \ll |t-s|^2 \log^2 p\ll |t-s|^{3/2}$$
 and the desired bound follows. We are left to consider the range $\vert t-s\vert > \frac{1}{\log^4 p}.$
To this end, we show that there exist absolute positive constants $C_1$ and $C_2$, such that for all $0\leq s<t\leq 1$, 
\begin{equation}\label{keytight}
\mathbb{E}\left\vert G_p(k,t)-G_p(k,s) \right\vert^{2} \leq C_1 |t-s|^2 +\frac{C_2}{p}.
\end{equation}
Indeed, if \eqref{keytight} holds, then since $\frac{1}{p^{2/3}} \ll \frac{1}{\log^4 p},$ in the range $\vert t-s\vert > \frac{1}{\log^4 p}$ we have
$$ \mathbb{E}\left\vert G_p(k,t)-G_p(k,s) \right\vert^{2} \leq C_3|t-s|^{3/2}$$ for some absolute constant $C_3>0.$\\
We now write
$$\left\vert G_p(k,t)-G_p(k,s)\right\vert^2 =  \left\vert \sum_{\vert m\vert \leq (p-1)/2} \left(\frac{k-m}{p}\right)a_{m,p}(s,t)\right\vert^2$$ where 
$a_{m,p}(s,t) = \alpha_p(m;s)-\alpha_p(m;t).$

%$$ G_p(k,t)-G_p(k,s)=\sum_{\vert m\vert \leq (p-1)/2} \left(\frac{k-m}{p}\right) \left\{\frac{e(t)-1}{p(e(\frac{m+t}{p})-1)}-\frac{e(s)-1}{p(e(\frac{m+s}{p})-1)}\right\}$$ and note that
%$$\left\vert G_p(k,t)-G_p(k,s)\right\vert^2 =  \left\vert \sum_{\vert m\vert \leq (p-1)/2} \left(\frac{k-m}{p}\right)a_m(s,t)\right\vert^2$$ where 
%\[ a_m(s,t) = \alpha(m;s)-\alpha(m;t).\]
%$$a_m=\begin{cases}\frac{e(t)-1}{p(e(\frac{m+t}{p})-1)}-\frac{e(s)-1}{p(e(\frac{m+s}{p})-1)} &\textrm{ if } (k-m,p)=1, \\
%0 &\textrm{      otherwise.} \end{cases}$$ 
By the approximation \eqref{ApproxCoefProcess} we can write \begin{equation}\label{alpham}a_{m,p}(s,t) = g_m(s)-g_m(t) + O \left(\frac{1}{p}\right)\end{equation} with $g_m(t):= \frac{e(t)-1}{2i \pi (m-t)}$. For any $\vert m\vert \leq (p-1)/2$, we have
\begin{equation}\label{bound-derivg}\max_{t\in [0,1]} \vert g'_m(t)\vert \leq C_4\min\left(\frac{1}{m},1\right) \end{equation} for some absolute constant $C_4>0$. 
Using \eqref{alpham} together with \eqref{bound-derivg} and applying the mean-value theorem, we deduce the bound
\begin{equation}\label{boundan} \vert a_{m,p}(s,t)\vert \leq C_4\vert t-s\vert \min\left(\frac{1}{m},1\right)+\frac{C_5}{p} \end{equation} for some absolute constant $C_5>0.$ Taking the expectation, using Lemma \ref{quadsumLegendre} and the Cauchy-Schwarz inequality, we arrive at 
\begin{align*}\mathbb{E}\left(\left\vert G_p(k,t)-G_p(k,s)\right\vert^2\right)& = \frac{1}{p}\sum_{k=0}^{p-1}\left\vert \sum_{\vert m\vert \leq (p-1)/2} \left(\frac{k-m}{p}\right)a_{m,p}(s,t)\right\vert^2 \\
& \leq \sum_{\vert m\vert \leq (p-1)/2} \vert a_{m,p}(s,t)\vert^2 
 + \frac{1}{p} \sum_{\substack{\vert m\vert, \vert m'\vert \leq (p-1)/2 \\ m \neq m'}} \vert a_{m,p}(s,t)\vert \vert a_{m',p}(s,t)\vert  \\  & \leq 2 \sum_{\vert m\vert \leq (p-1)/2} \vert a_{m,p}(s,t)\vert^2.\end{align*} Applying \eqref{boundan} gives
\begin{align*} \sum_{\vert m\vert \leq (p-1)/2} \vert a_{m,p}(s,t)\vert^2 &\ll \sum_{1\leq m \leq p}\frac{\vert t-s\vert^2}{m^2}+\frac{1}{p^2}\sum_{\vert m\vert \leq (p-1)/2}1 \\
& \ll \vert t-s\vert^2 + 1/p\end{align*}
and \eqref{keytight} follows. This concludes the proof.

\end{proof}
Theorem \ref{convprocess} now immediately follows by combining Proposition \ref{ProMomentsFiniteDistribution} together with Lemma \ref{lemlarge}.

%%%%%%%%%%%%%%%%%%%%%%%%%%%%%%%%%%%%%%%%%%%%%%%%%%%%%%%%%%%%%%%%%%%%%%%%%%%%%%%%%%%%%%%%%%%%%%%%%%%%%%%%%%%%%%%%%%%%%%%%%%%%%%%%%%%%%%%%%%%%%%%%%%%

\section{$L^q$ norms and the distribution of values of Fekete polynomials: Proofs of Corollary \ref{cor:lqnorms} and Theorem \ref{main_convprocess}}\label{section-norms}
%\subsection{$L^q$ norms}
In this section we derive Corollary \ref{cor:lqnorms} and Theorem \ref{main_convprocess} as consequences of Theorem \ref{convprocess}.

\begin{proof}[Proof of Corollary \ref{cor:lqnorms}]For every $q>0,$ we consider a continuous functional on $C[0,1]$ given by
$$\ell_q(\phi)=\int_{0}^1 |\phi(t)|^q dt.$$ We have 
\[
\int_0^1| F_p(e(t))|^q dt = \frac 1p \sum_{k=0}^{p-1}    \int_0^1  \bigg| F_p\bigg( e\bigg( \frac{k+x}p\bigg) \bigg) \bigg|^q  dx
=  (p)^{q/2}\frac 1p \sum_{k=0}^{p-1}    \int_0^1  |G_p(k,t)|^q dx. 
\]
The result now immediately follows from Theorem \ref{convprocess}. 
\end{proof}
\begin{proof}[Proof of Theorem \ref{main_convprocess}] For every non-negative integers $k, j$ we consider a continuous functional on $C[0,1]$ given by
$$\ell_{k, j}(\phi)=\int_{0}^1 \phi(t)^k \overline{\phi(t)^j} dt.$$
Then we have 
$$ \int_0^1\left(\frac{F_p(e(t))}{F_p(\zeta_p)}\right)^k\left(\overline{\frac{F_p(e(t))}{F_p(\zeta_p)}}\right)^j dt = \frac 1p \sum_{k=0}^{p-1}    \int_0^1 \Bigg(\frac{F_p\big( e\big( \frac{k+x}p\big) \big)}{F_p(\zeta_p)}\Bigg)^k\Bigg(\overline{\frac{F_p\big( e\big( \frac{k+x}p\big) \big)}{F_p(\zeta_p)}}\Bigg)^jdx.$$
Therefore, it follows from Theorem \ref{convprocess} that 
$$ \lim_{p\to \infty}\int_0^1\left(\frac{F_p(e(t))}{F_p(\zeta_p)}\right)^k\left(\overline{\frac{F_p(e(t))}{F_p(\zeta_p)}}\right)^j dt= \mathbb{E}\left(\int_0^1 G_{\X}(t)^k \overline{G_{\X}(t)}^jdt\right)= \ex\left(G_{\X}(\theta)^k \overline{G_{\X}(\theta)}^j\right),$$
where $\theta$ is a random variable uniformly distributed on $[0,1]$. The result follows since $G_{\X}(\theta)$ is an absolutely continuous random variable by Lemma \ref{approx-distribution}.
\end{proof}
%\begin{proof}[Proof of Theorem \ref{main_convprocess}] The proof that
%\[ \lim_{p\rightarrow \infty} \text{m} \left(\theta\in [0,1]:\frac{F_p(e^{i\theta})}{F_p(\zeta_p)}\in U\right) = \ex \left\{ \mathbb{P} \left(G_{\mathbb{X}}(t)\in U\right)\right\}\]
%immediately follows from Theorem \ref{convprocess} (weak convergence) by applying Skorokhod's representation theorem 6.7 of \cite{Bill}, together with the Portmanteau theorem 2.1 of \cite{Bill} in conjunction with our Lemma \ref{approx-distribution}.
%\end{proof}}

%%%%%%%%%%%%%%%%%%%%%%%%%%%%%%%%%%%%%%%%%%%%%%%%%%%%%%%%%%%%%%%%%%%%%%%%%%%%%%%%%%%%%%%%%%%%%%%%%%%%%%%%%%%%%%%%%%%%%%%%%%%%%%%%%%%%%%%%%%%%%%%%%%%

\section{The Mahler measure of Fekete polynomials: Proof of Theorem \ref{cor: Mahnorms}}\label{computemahler}
As was mentioned in the introduction, the main difficulty now is that the functional on ${C}[0,1]$ defined by $\ell(\phi)=\int_{0}^1\log (|\phi(t)|)dt$ is not continuous so our previous results do not directly apply.
To circumvent this problem, we define
\begin{equation}\label{defHp}
H_p(k,t):= \frac{2\pi i}{(e(t)-1)} G_p(k,t)= \sum_{|m|\leq (p-1)/2} \bigg(\frac{k-m}p\bigg)  \frac{2\pi i} {p(e( \frac{m-t}p)-1) },
\end{equation}% such that 
 %we have $|G_p(k,t)|=|\frac{(e(t)-1)}{2\pi} H_{p}(k,t)|$ 
 and similarly 
 \begin{equation}\label{defHX}
 H_{\X}(t)=\frac{2\pi i}{(e(t)-1)} G_{\X}(t).
 \end{equation}Fix $\varepsilon>0$. First, we note that $H_p(k, t)$ is continuous on $[\varepsilon,1-\varepsilon]$ and $H_{\X}( t)$ is almost surely continuous on this interval by Lemma \ref{lem:smooth-process}.
Therefore, it follows from Theorem \ref{convprocess} that the sequence  of random processes $(H_p(k, t))_{t\in [\varepsilon,1-\varepsilon]}$  converges in law (in the space $C[\varepsilon,1-\varepsilon]$) to the process  $(H_{\X}(t))_{t\in [\varepsilon, 1-\varepsilon]}$.
We now consider the following functional $$\widetilde{\ell}_{\varepsilon}(\phi)=\int_{\varepsilon}^{1-\varepsilon}\log (|\phi(t)|)\mathbbm{1}_{|\phi(t)|\ge \varepsilon}dt.$$ Since this is a continuous functional on ${C}[\varepsilon,1-\varepsilon]$ we deduce that 
$$
\lim_{p\to\infty}\frac{1}{p}\sum_{0\le k\le p-1}\int_{\varepsilon}^{1-\varepsilon}\log (|{H}_p(k,t)|)\mathbbm{1}_{|{H}_p(k,t)|\ge \varepsilon}dt=\int_{\varepsilon}^{1-\varepsilon}\mathbb{E}\big(\log (|{H}_{\mathbb{X}}(t)|)\mathbbm{1}_{|{H}_{\mathbb{X}}(t)|\ge \varepsilon}\big)dt.
$$ In order to finish the proof of Theorem \ref{cor: Mahnorms}, we need to prove that the remaining contributions are small as $\varepsilon\to 0$, which is the content of Lemmas \ref{lem: log-trunc} and \ref{lem: logtrunc-01} below.
%Fix $\varepsilon>0$ and consider the following functional $$\ell_{\epsilon}(\phi)=\int_{0}^1\log (|\phi(t)|)\mathbbm{1}_{|\phi(t)|\ge \varepsilon}dt.$$
%This is a continuous functional on ${C}[0,1]$ and so in view of Lemma \ref{lemlarge} and Proposition \ref{ProMomentsFiniteDistribution} we have convergence
%\begin{equation}\label{truncated-convergence}
%\lim_{p\to\infty}\frac{1}{p}\sum_{0\le k\le p-1}\int_{0}^1\log (|{G}_p(k,t)|)\mathbbm{1}_{|{G}_p(k,t)|\ge \varepsilon}dt=\int_{0}^1\mathbb{E}\log (|{G}_{\mathbb{X}}(t)|)\mathbbm{1}_{|{G}_{\mathbb{X}}(t)|\ge \varepsilon}dt.
%\end{equation}
We now introduce the following approximation
\[
  \widetilde{G}_p(k, t):= \sum_{|m|\leq (p-1)/2} \bigg(\frac{k-m}p\bigg)  \frac{e(t)-1} {2\pi i (m-t) }
\] and, for technical reasons, perform our arguments in Lemma \ref{lem: log-trunc} with $\widetilde{G}_p(k,t).$% In order to do so we need the following deterministic analogue of Lemma \ref{SecondMomentApprox}. %\begin{lem}\label{SecondMomentApproxdet}
%Let $p$ be a large prime. Then uniformly for $t\in [0,1]$ we have
%$$\frac{1}{p}\sum_{k=0}^{p-1} \left\vert G_{p}(k,t)-\widetilde{G}_{p}(k,t)\right|^2\ll \frac{1}{p}, $$
%where the implicit constant is absolute.
%\end{lem}

%\begin{proof}
%Let $$E(m,t):= \frac{e(t)-1} {p( e( \frac{m-t}p)-1) }-\frac{e(t)-1}{2\pi i (m-t)}.$$ Notice that by \eqref{ApproxCoefProcess}, we uniformly have $E(m,t) \ll 1/p.$ Expanding the squares we are left to bound
%$$  \sum_{|m_1,m_2|\leq \frac{p-1}{2}} \bigg( \frac{(k-m_1)(k-m_2)}p\bigg)  E(m_1,t)\overline{E(m_2,t)}  .$$ Summing over $k$, changing the order of summation and using Lemma \ref{quadsumLegendre} we see that the contribution of off-diagonal terms ($m_1 \neq m_2$) is $\ll 1/p$. Again by Lemma \ref{quadsumLegendre} and \eqref{ApproxCoefProcess}, the contribution of the diagonal terms is also $\ll 1/p.$
%\end{proof} % It immediately implies
%\begin{cor}\label{exceptionalsetapprox}
%$$ \frac{1}{p} \# \left\{0\leq k \leq p-1, \lVert   G_{p}(k,t)-\widetilde{G}_{p}(k,t)\rVert_{\infty} > \frac{1}{p^{1/3}}\right\}
% \ll p^{-1/3}.$$ \end{cor}
% Let us recall that for all $|m|\leq (p-1)/2$ and $t\in [0, 1]$ we have uniformly
%\begin{equation*}
%\frac{e(t)-1}{p( e( \frac{m-t}p)-1)}= \frac{e(t)-1}{2\pi i (m-t)}+O\left(\frac{1}{p}\right).
%\end{equation*} %Hence if $k$ is such that $$\left\vert G_{p}(k,t)-\widetilde{G}_{p}(k,t)\right\vert >\frac{1}{p^{1/3}}$$ we use the trivial approximation 
%$$ G_{p}(k,t)=\widetilde{G}_{p}(k,t) + O\left(1\right). $$
To this end, we further define 
\begin{equation}\label{defHptilde}
\widetilde{H}_p(k,t)=\sum_{|m|\leq (p-1)/2} \bigg(\frac{k-m}p\bigg)  \frac{1} { (m-t) }
\end{equation} so that $|\widetilde{G}_p(k,t)|=|\frac{(e(t)-1)}{2\pi} \widetilde{H}_{p}(k,t)|$. 
%In the same way, we define
%\begin{equation}\label{defHp}
%H_p(k,t)= \sum_{|m|\leq (p-1)/2} \bigg(\frac{k-m}p\bigg)  \frac{2\pi i} {p(e( \frac{m-t}p)-1) },
%\end{equation} such that 
% we have $|G_p(k,t)|=|\frac{(e(t)-1)}{2\pi} H_{p}(k,t)|$. 
We now observe that $H_p(k,t)$ is a real-valued function of $t\in [0,1]$ and proceed to show that $H_p(k,t)$ and $\widetilde{H}_p(k,t)$ are close  in $C^{2}$ -topology.
\begin{lem}\label{lem-deriv}
For any $0\leq k \leq p-1$, the following bound holds 
 $$ \lVert   \widetilde{H}'_{p}(k,)-H'_{p}(k,)\rVert_{\infty} = O(1/p)$$ where $$ \lVert f\rVert_{\infty}:= \sup_{t\in (0,1)}\vert f(t)\vert.$$ Similarly, for the second derivatives we have
$$ \lVert   \widetilde{H}''_{p}(k,)-H''_{p}(k,)\rVert_{\infty} = O(1/p^2), \hspace{5mm} 0\leq k \leq p-1. $$

\end{lem}

\begin{proof}
 Upon differentiating \eqref{defHptilde} and \eqref{defHp}, we obtain
  $$ \widetilde{H}'_p(k,t) = \sum_{|m|\leq (p-1)/2} \bigg(\frac{k-m}p\bigg)  \frac{1} { (m-t)^2} $$ and
$$H'_p(k,t) = \sum_{|m|\leq (p-1)/2} \bigg(\frac{k-m}p\bigg)  \frac{(2\pi i)^2e( \frac{m-t}p)}{p^2( e( \frac{m-t}p)-1)^2}.$$ Factoring out $e(\frac{m-t}{2})$ from the denominator of this last identity yields
$$H'_p(k,t) = \sum_{|m|\leq (p-1)/2} \bigg(\frac{k-m}p\bigg)  \frac{\pi^2}{p^2 \sin^2\left(\frac{\pi}{p}(m-t)\right)}.   $$
Using Taylor expansion, we have that, uniformly over $t\in [0, 1]$
\begin{equation}\label{approx-deriv1}
\frac{\pi^2}{p^2 \sin^2\left(\frac{\pi}{p}(m-t)\right)}= \frac{1}{(m-t)^2}+O\left(\frac{1}{p^2}\right).\end{equation}
Summing over $m$ and using \eqref{approx-deriv1} concludes the proof of the first part.
We now differentiate to end up with
$$ \widetilde{H}''_p(k,t) = \sum_{|m|\leq (p-1)/2} \bigg(\frac{k-m}p\bigg)  \frac{2} { (m-t)^3} $$ and
$$ H''_p(k,t) =  \frac{2\pi^3}{p^3} \sum_{|m|\leq (p-1)/2} \frac{\cos\left(\frac{\pi}{p}(m-t)\right)}{\sin^3\left(\frac{\pi}{p}(m-t)\right)}.$$ By Taylor expansion, uniformly over $t\in [0, 1]$ we have
\begin{equation}\label{approx-deriv2}
\frac{\pi^3}{p^3} \frac{\cos\left(\frac{\pi}{p}(m-t)\right)}{\sin^3\left(\frac{\pi}{p}(m-t)\right)}= \frac{1}{(m-t)^3}+O\left(\frac{m-t}{p^4}\right).\end{equation} Summing over $m$ and using \eqref{approx-deriv2} finishes the proof.

\end{proof}

%Let us first remark that for $k\neq 0,1$ we can write %$|\widetilde{G}_p(k,t)|=|\frac{(e(t)-1)}{2\pi} H_{p}(k,t)|$ where 
%$$H_{p}(k,t)=\frac{1}{t}+\frac{\delta_0}{1-t}+\sum_{\substack{m\ne 0,-1 \\ |m|\le \frac{p-1}{2}}}\frac{\delta_m}{m+t},$$
%with $\delta_i=\pm 1$ depending on $k$. 
%In order to perform our arguments on  the ``random'' side we introduce the following set of auxiliary functions. For $\left\{\delta_m\right\}_{m\in \mathbb{Z}} \in \left\{\pm 1\right\}^{\mathbb{Z}}$ we define for $t\in (0,1)$ the function
%\begin{equation}\label{defapproxseries}
%H_{\X}(t)= \sum_{m \in \mathbb{Z}} \frac{\delta_m}{m-t}.
%\end{equation} 
The following technical result plays a crucial role in the proof of Theorem \ref{cor: Mahnorms}.
\begin{lem}\label{lem: log-trunc}
\leavevmode For any $0\leq k \leq p-1$, we have
 $$\int_{0}^1\log (|{\widetilde{H}}_{p}(k,t)|)\mathbbm{1}_{|\widetilde{H}_{p}(k,t)|\le \varepsilon}dt =O(\varepsilon^{6/25})  \text{ and }\int_{0}^1\log (|H_p(k,t)|)\mathbbm{1}_{|H_p(k,t)|\le \varepsilon}dt =O(\varepsilon^{6/25}).$$

%\Item   $$ \int_{0}^1\log (|H_p(k,t)|)\mathbbm{1}_{|H_p(k,t)|\le \varepsilon}dt=O(\varepsilon^{6/25}).$$ \end{enumerate}
 Moreover the same result holds for $H_{\X}$ almost surely, namely
 $$ \int_{0}^1\log (|H_{\X}(t)|)\mathbbm{1}_{|H_{\X}(t)|\le \varepsilon}dt=O(\varepsilon^{6/25}).$$
\end{lem}
\begin{proof} We first prove the result in the case of the function $\widetilde{H}_p(k,t)$. The main idea is to show that either $|\widetilde{H}_{p}'(k,t)(t)|\gg 1,$ or $|\widetilde{H}_{p,k}''(t)|\gg 1.$ Then the same result for $H_{p}(k,t)$ follows from Lemma \ref{lem-deriv}. To do so, we split the discussion into three cases depending on the value of $k$. \begin{case} Here we consider the case $2\leq k \leq p-1.$  \end{case}
In this case we can write $$\widetilde{H}_{p}(k,t)=-\left(\frac{k}{p}\right)\Bigg(\frac{1}{t}+\frac{\delta_0}{1-t}+\sum_{\substack{m\ne 0,-1 \\ |m|\le \frac{p-1}{2}}}\frac{\delta_m}{m+t}\Bigg),$$
where $\delta_i=\pm 1,$ which depend on $k.$ Our argument covers all possible choices of $\pm 1$ coefficients and we can therefore from now on remove the factor $-(k/p)$, drop the condition on $k$ and hence consider 
$$\widetilde{H}_{p}(t):= -\left(\frac{k}{p}\right)\widetilde{H}_{p}(k,t)=\frac{1}{t}+\frac{\delta_0}{1-t}+\sum_{\substack{m\ne 0,-1 \\ |m|\le \frac{p-1}{2}}}\frac{\delta_m}{m+t},$$
for some $\delta_i=\pm 1.$
We first observe that each of the integrals is well-defined as an integral of a rational function with logarithmic singularities being integrable around poles and zeros. 
 We distinguish two subcases depending on the value of $\delta_0.$
 \begin{itemize}
\item In the case when $\delta_0=-1$, we have the bound
$$\widetilde{H}'_{p}(t)\le -\frac{1}{t^2}-\frac{1}{(1-t)^2}+\sum_{m\ne 0,-1,|m|\le p-1/2}\frac{1}{(t+m)^2}< -8+2\zeta(2)=c_1<0.$$
Since  $\lim_{t\to 0^+}\widetilde{H}_{p}(t)=-\lim_{t\to 1^-}\widetilde{H}_{p}(t)=\infty,$ we let $a, b\in [0,1]$ to be the unique solutions to $\widetilde{H}_{p}(a)=-\widetilde{H}_{p}(b)=\varepsilon.$ For a sufficiently small $\varepsilon>0,$ we have 
$|\log |\widetilde{H}_{p}(t)| \mathbbm{1}_{|\widetilde{H}_{p}(t)|\le \varepsilon}|\le\frac{1}{\sqrt{|\widetilde{H}_{p}(t)|}}$ whenever $\widetilde{H}_{p}(t)\ne 0$ and consequently 
$$\left|\int_{0}^1 \log |\widetilde{H}_{p}(t)| \mathbbm{1}_{|\widetilde{H}_{p}(t)|\le \varepsilon}dt\right|\le \int_{a}^b \frac{1}{-\widetilde{H}'_{p}(t)}\frac{-\widetilde{H}'_{p}(t)}{\sqrt{|\widetilde{H}_{p}(t)|}}dt\le \frac{4\sqrt{\varepsilon}}{-c_1} =O(\varepsilon^{6/25}).$$
\item In the case where $\delta_0=1,$ we have the bound
$$\frac{\widetilde{H}''_{p}(t)}{2}\ge\frac{1}{t^3}+\frac{1}{(1-t)^3}-\sum_{m\ne 0,-1,|m|\le p-1/2}\frac{1}{(t+m)^3}> 16-2\zeta(3)=c_2>0.$$
 We observe that $\lim_{t\to 0^+}\widetilde{H}_{p}(t)=\lim_{t\to 1^-}\widetilde{H}_{p}(t)=\infty$ and consider the following three possibilities:
\begin{enumerate}
 \item $\widetilde{H}_{p}(t)>\varepsilon$ for all $t\in [0,1].$ In this case our conclusion follows trivially.
\item There exist $0<t_1<t_2<t_3<t_4<1$ such that  $\widetilde{H}_{p}(t_1)=-\widetilde{H}_{p}(t_2)=-\widetilde{H}_{p}(t_3)=\widetilde{H}_{p}(t_4)=\varepsilon.$
\item There exist $0<t_1<t_2<1$ such that  $\widetilde{H}_{p}(t_1)=\widetilde{H}_{p}(t_2)=\varepsilon$ and $\widetilde{H}_{p}(t)> -\varepsilon$ for all $t\in[0,1].$
\end{enumerate}
Since the arguments in the second and third cases are identical, we focus on the second case. By symmetry, it suffices to show that   $$\int_{t_3}^{t_4} \log |\widetilde{H}_{p}(t)|dt=O(\varepsilon^{6/25}).$$ To this end we note that for sufficiently small $\varepsilon>0,$ $$|\log |\widetilde{H}_{p}(t)|\mathbbm{1}_{|\widetilde{H}_{p}(t)|\le \varepsilon}\vert\le\frac{1}{|\widetilde{H}_{p}(t)|^{1/100}}$$and $\widetilde{H}'_{p}(t)>0$ for all $t\in [t_3,t_4].$  Applying H\"{o}lder's inequality, we get
\begin{align*}
\left|\int_{t_3}^{t_4} \log |\widetilde{H}_{p}(t)| \mathbbm{1}_{|\widetilde{H}_{p}(t)|\le \varepsilon} dt\right|&\le \int_{t_3}^{t_4} \frac{1}{\widetilde{H}'_{p}(t)^{1/4}}\frac{\widetilde{H}'_{p}(t)^{1/4}}{|\widetilde{H}_{p}(t)|^{1/100}}dt\\&\le  \left(\int_{t_3}^{t_4} \frac{1}{\widetilde{H}'_{p}(t)^{1/3}}dt\right)^{3/4}\left(\int_{t_3}^{t_4} \frac{\widetilde{H}'_{p}(t)}{|\widetilde{H}_{p}(t)|^{1/25}}dt\right)^{1/4}.
\end{align*}
Integrating the second integral directly, we get an acceptable contribution of $O(\varepsilon^{6/25}).$ Since $\widetilde{H}''_{p}(t)\ge 2c_2,$ we have a linear minorant $\widetilde{H}'_{p}(t)\ge 2c_2(t-t_3)+\widetilde{H}'_{p}(t_3)\ge 2c_2(t-t_3) $ and consequently

$$ \int_{t_3}^{t_4} \frac{1}{\widetilde{H}'_{p}(t)^{1/3}}dt\ll \int_{t_3}^{t_4} \frac{1}{(t-t_3)^{1/3}}dt \ll (t_4-t_3)^{2/3}=O(1).$$
This concludes the proof of Case $1$.
\end{itemize}
\begin{case} Here we focus on the case $k=0.$  
\end{case}
In this case, the term corresponding to $m=0$ in the sum defining $\widetilde{H}_{p}(0,t)$ vanishes and we can write  
$$ \widetilde{H}_{p}(0,t)=\frac{1}{1-t}+\frac{\delta_0}{1+t}+\sum_{\substack{m\ne 0,-1,1 \\ |m|\le \frac{p-1}{2}}}\frac{\delta_m}{m+t},$$
with $\delta_i=\pm 1$.
We again distinguish two cases depending on the value of $\delta_0.$
 \begin{itemize}
\item If $\delta_0=-1$, we have the bound
\begin{align*} \widetilde{H}'_{p}(0,t) & \ge \frac{1}{(1-t)^2}+\frac{1}{(1+t)^2}-\sum_{\substack{|m|\le (p-1)/2 \\ m\neq 0,-1,1}} \frac{1}{(t+m)^2}. \end{align*} We consider two subcases. If $ t >1/2$, we have 
\begin{align*}
    \widetilde{H}'_{p}(0,t) & \ge 
\frac{1}{(1-t)^2}+\frac{1}{(1+t)^2}-\sum_{m\ge 2}\left\{\frac{1}{(t+m)^2}+\frac{1}{(m-t)^2}\right\} \\ &\ge 4-2\zeta(2)=c_{3,1} >0.\end{align*} On the other hand, if $0 \le t\le 1/2$ we get \begin{align*}
    \widetilde{H}'_{p}(0,t) & \ge 
\frac{1}{(1-t)^2}+\frac{1}{(1+t)^2}-\sum_{m\ge 2}\left\{\frac{1}{(1/2+m)^2}+\frac{1}{(m-1/2)^2}\right\} \\ &\ge
2-4\sum_{m \geq 2}\left\{\frac{1}{(2m+1)^2}+\frac{1}{(2m-1)^2} \right\}\ge
 2-4\left(\frac{\pi^2}{4}-2\right)=c_{3,2} >0.\end{align*}
%& \ge  2-\sum_{m\ge 2}\left\{\frac{1}{(t+m)^2}+\frac{1}{(m-t)^2}\right\} \ge   2-(2\zeta(2)-1-1/4)=c_3>0.\end{align*}
\item If $\delta_0=1,$ then
\begin{align*} \frac{\widetilde{H}''_{p}(0,t)}{2} & \ge \frac{1}{(1-t)^3}+\frac{1}{(1+t)^3}-\sum_{\substack{|m|\le (p-1)/2 \\ m\neq 0,-1,1}} \frac{1}{(t+m)^3} \\
& \ge  2-\sum_{m\ge 2}\left\{\frac{1}{(t+m)^3}+\frac{1}{(m-t)^3}\right\}  \ge   2-(2\zeta(3)-1)=c_4>0.\end{align*}
\end{itemize}
The rest of the proof is identical to Case $1$. 
\begin{case} We finally consider $k=1.$  
\end{case}
We remark that in the sum defining $\widetilde{G}_p(1,t)$ the term corresponding to $m=1$ vanishes. Hence we write 
$$\widetilde{H}_{p}(1,t)=-\frac{1}{t}+\frac{\delta_0}{2-t}+\sum_{\substack{m\ne 0,-1,-2 \\ |m|\le \frac{p-1}{2}}}\frac{\delta_m}{m+t},$$
with $\delta_i=\pm 1$. We distinguish two cases depending on the value of $\delta_0.$
 \begin{itemize}
 \item If $\delta_0=1$, we have the bound
$$\widetilde{H}'_{p}(1,t)\ge \frac{1}{t^2}+\frac{1}{(2-t)^2}-\sum_{\substack{|m|\le (p-1)/2 \\ m \neq 0,-1,-2}}\frac{1}{(t+m)^2}.$$ We again split the treatment into two subcases. If $0 \le t\leq 1/2$, we have 
$$ \frac{1}{t^2}+\frac{1}{(2-t)^2}-\sum_{\substack{|m|\le (p-1)/2 \\ m \neq 0,-1,-2}}\frac{1}{(t+m)^2} \geq 4-2\zeta(2)=c_{5,1} >0.$$ On the other hand, if $t>1/2$ we get \begin{align*}
\widetilde{H}'_p(1,t) & \ge  \frac{1}{t^2}+\frac{1}{(2-t)^2}-\sum_{m \ge 1}\frac{1}{(m+1/2)^2}-\sum_{m \geq 3} \frac{1}{(m-1)^2}   \\ 
& \ge
2-\left\{4\left(\frac{\pi^2}{8}-1\right)+\zeta(2)-1\right\}=c_{5,2}>0.\end{align*}
 \item If $\delta_0=-1,$ then
$$\frac{\widetilde{H}''_{p}(1,t)}{2}\le-\frac{1}{t^3}-\frac{1}{(2-t)^3}+\sum_{\substack{|m|\le (p-1)/2 \\ m\neq 0,-1,-2}}\frac{1}{(t+m)^3}<- 2+(2\zeta(3)-1)=-c_4<0.$$  
\end{itemize}
The rest of the proof is again identical to Case $1$. \\
We now briefly comment on the remaining parts of the lemma. 
Using Lemma \ref{lem-deriv} and taking $p$ large enough we can assume that the first two derivatives of $\widetilde{H}_p(k,t)$ and $H_p(k,t)$ are $\nu$-close for a small and suitably chosen $\nu>0.$ It follows that the graph of $\widetilde{H}_p(k,t)$ and $H_p(k,t)$ have a ``similar" shape and therefore we can perform the identical arguments for $H_p(k,t)$ as was done for $\widetilde{H}_p(k,t)$ to conclude the result. \\
Notice that our argument relies only on the first terms in the expansion of the derivatives of $\widetilde{H}_p(k,t)$, the other terms being bounded trivially after extending the truncation to its full series. Furthermore, by Lemma \ref{lem:smooth-process} $H_{\X}$ is almost surely a real-analytic function on $(0,1)$ and the argument goes over as in Case $1$ for $\widetilde{H}_p(k,t)$, mutatis mutandis, for the function $H_{\X}$.
\end{proof}
We are now ready to prove our main statement about truncating the logarithmic integral.
\begin{lem}\label{lem: logtrunc-01}
Let $p$ be a large prime number and  $0<\varepsilon<1/2$ be a real number. Then we have 
\begin{equation}\label{eq: logtrunc-01-1}
\frac{1}{p} \sum_{0\leq k\leq p-1} \left(\int_{0}^{\varepsilon} + \int_{1-\varepsilon}^1\right) \log|H_p(k,t)|\mathbbm{1}_{|H_{p}(k,t)|\ge \varepsilon} dt\ll \varepsilon \log(1/\varepsilon).
\end{equation}
Moreover, a similar estimate holds in the random case, namely 
\begin{equation}\label{eq: logtrunc-01-2}
    \left(\int_{0}^{\varepsilon}+ \int_{1-\varepsilon}^1\right)\mathbb{E}\big(\log (|{H}_{\mathbb{X}}(t)|)\mathbbm{1}_{|{H}_{\mathbb{X}}(t)|\ge \varepsilon}\big)dt\ll \varepsilon \log(1/\varepsilon).
    \end{equation}
\end{lem}
\begin{proof} We will only prove \eqref{eq: logtrunc-01-1} since the proof of \eqref{eq: logtrunc-01-2} is similar. Moreover, we shall only consider the first part involving the integral  $\int_{0}^{\varepsilon}$ since an analogous bound for the second part (over the  integral $\int_{1-\varepsilon}^1$) follows along the same lines.

Let $L_p(k, t):= t H_p(k, t).$ We shall split the inner integral from $0$ to $\varepsilon$ in \eqref{eq: logtrunc-01-1} into two parts depending on whether $|L_p(k, t)|$ is ``large" or not. More precisely we have 
\begin{equation}\label{eq: intlogsplit}
\begin{aligned}
\int_{0}^{\varepsilon} \log|H_p(k,t)|\mathbbm{1}_{|H_{p}(k,t)|\geq \varepsilon} dt &= \int_{0}^{\varepsilon} \log|H_p(k,t)|\mathbbm{1}_{|H_{p}(k,t)|\geq \varepsilon}\mathbbm{1}_{|L_{p}(k,t)|> 1} dt \\
&  \quad + \int_{0}^{\varepsilon} \log|H_p(k,t)|\mathbbm{1}_{|H_{p}(k,t)|\geq \varepsilon}\mathbbm{1}_{|L_{p}(k,t)|\leq 1} dt.
\end{aligned}
\end{equation}
We begin by estimating the second integral. First, note that 
$$ \int_{0}^{\varepsilon} \log|H_p(k,t)|\mathbbm{1}_{|H_{p}(k,t)|\geq \varepsilon}\mathbbm{1}_{|L_{p}(k,t)|\leq 1} dt \geq \log(\varepsilon) \int_{0}^{\varepsilon} \mathbbm{1}_{|H_{p}(k,t)|\geq \varepsilon}\mathbbm{1}_{|L_{p}(k,t)|\leq 1} dt \geq \varepsilon \log(\varepsilon), 
$$
since $\log\varepsilon< 0$. On the other hand we have 
\begin{align*}
&\int_{0}^{\varepsilon} \log|H_p(k,t)|\mathbbm{1}_{|H_{p}(k,t)|\geq \varepsilon}\mathbbm{1}_{|L_{p}(k,t)|\leq 1} dt \\ 
= & \int_{0}^{\varepsilon} \log|L_p(k,t)|\mathbbm{1}_{|H_{p}(k,t)|\geq \varepsilon}\mathbbm{1}_{|L_{p}(k,t)|\leq 1} dt 
+O\left(\int_{0}^{\varepsilon}| \log t|dt\right)\\
\leq  & \ C_1\varepsilon \log (1/\varepsilon), 
\end{align*}
for some positive constant $C_1$. Combining these estimates implies that 
\begin{equation}\label{eq: intlog_small}
\int_{0}^{\varepsilon} \log|H_p(k,t)|\mathbbm{1}_{|H_{p}(k,t)|\geq \varepsilon}\mathbbm{1}_{|L_{p}(k,t)|\leq 1} dt \ll 
\varepsilon \log(1/\varepsilon).
\end{equation}
We now investigate the first integral on the right hand side of \eqref{eq: intlogsplit}. Since $|H_p(k, t)|\geq |L_p(k,t)|$ for all $t\in [0, 1]$, we get for $\varepsilon <1$ 
%\begin{equation}\label{eq: intlog_drop}
\begin{align}\label{eq:intlog_drop}\int_{0}^{\varepsilon} \log|H_p(k,t)|\mathbbm{1}_{|H_{p}(k,t)|\geq \varepsilon}\mathbbm{1}_{|L_{p}(k,t)|>1} dt &=\int_{0}^{\varepsilon} \log|H_p(k,t)|\mathbbm{1}_{|L_{p}(k,t)|> 1} dt \nonumber \\
& = \int_{0}^{\varepsilon} \log|L_p(k,t)|\mathbbm{1}_{|L_{p}(k,t)|> 1} dt + O\left(\varepsilon \log(1/\varepsilon)\right).
\end{align}
%\end{equation}
Note that the integral on the main term of this last estimate is non-negative. Furthermore, it follows from \eqref{ApproxCoefProcess} that uniformly for $t\in[0,1/2]$
$$L_p(k, t)=t \sum_{|m|\leq (p-1)/2} \left(\frac{k-m}{p}\right) \frac{1}{m-t} +O(1)= t\sum_{1\leq |m|\leq (p-1)/2} \left(\frac{k-m}{p}\right)\frac{1}{m} +O(1).$$
Therefore, there exists a positive constant $C_2$ such that 
$$ \max_{t\in [0, 1/2]} |L_p(k, t)|\leq 
\Big|\sum_{1\leq |m|\leq (p-1)/2} \left(\frac{k-m}{p}\right)\frac{1}{m}\Big|+ C_2 \leq  \Big|\sum_{1\leq |m|\leq (p-1)/2} \left(\frac{k-m}{p}\right)\frac{1}{m}\Big|^2+ C_3,$$
where $C_3=C_2+1.$
Hence, we deduce that 
\begin{align}\label{eq: intlog_large}
0 & \leq \int_{0}^{\varepsilon} \log|L_p(k,t)|\mathbbm{1}_{|L_{p}(k,t)|>1} dt \nonumber \\
& \leq \varepsilon \log\left(\Big|\sum_{1\leq |m|\leq (p-1)/2} \left(\frac{k-m}{p}\right)\frac{1}{m}\Big|^2+ C_3\right).
\end{align}
Since the function $x\to \log(x+C_3)$ is concave on $[0, \infty)$ it follows from Jensen's inequality that 
\begin{equation}\label{eq: ConcaveJenssen}
\begin{aligned}
    &\frac{1}{p}\sum_{k=0}^{p-1}\log\left(\Big|\sum_{1\leq |m|\leq (p-1)/2} \left(\frac{k-m}{p}\right)\frac{1}{m}\Big|^2+ C_3\right) \\
    &\leq \log\left(\frac{1}{p}\sum_{k=0}^{p-1}\Big|\sum_{1\leq |m|\leq (p-1)/2} \left(\frac{k-m}{p}\right)\frac{1}{m}\Big|^2+C_3\right).
\end{aligned}
\end{equation}
Finally, by Lemma \ref{quadsumLegendre} we obtain
\begin{align*}
\frac{1}{p}\sum_{k=0}^{p-1}\Big|\sum_{1\leq |m|\leq (p-1)/2} \left(\frac{k-m}{p}\right)\frac{1}{m}\Big|^2& = \sum_{1\leq |m_1|, |m_2|\leq (p-1)/2} \frac{1}{m_1m_2} \frac{1}{p} \sum_{k=0}^{p-1} \left(\frac{k-m_1}{p}\right)\left(\frac{k-m_2}{p}\right)\\
& \ll \sum_{1\leq |m|\leq (p-1)/2} \frac{1}{m^2} + \frac{1}{p} \left(\sum_{1\leq |m|\leq (p-1)/2} \frac{1}{m}\right)^2 \ll 1.
\end{align*}
Combining this with the estimates  \eqref{eq:intlog_drop}, \eqref{eq: intlog_large} and \eqref{eq: ConcaveJenssen} yields
$$
\frac{1}{p}\sum_{k=0}^{p-1}\int_{0}^{\varepsilon} \log|H_p(k,t)|\mathbbm{1}_{|H_{p}(k,t)|\geq \varepsilon}\mathbbm{1}_{|L_{p}(k,t)|>1} dt \ll \varepsilon\log(1/\varepsilon). 
$$
This together with \eqref{eq: intlogsplit} and \eqref{eq: intlog_small} completes the proof.
\end{proof}

\begin{proof}[Proof of Theorem \ref{cor: Mahnorms}]
%Fix $\varepsilon>0$.

%First, we note that $H_p(k, t)$ is continuous on $[\varepsilon,1-\varepsilon]$, $H_{\X}( t)$ is almost surely continuous on this interval and record that
%$$ G_p(k, t)= \frac{(e(t)-1)}{2\pi i } H_p(k, t), \quad \  G_{\X}(t)= \frac{(e(t)-1)}{2\pi i } H_{\X}(t). 
%$$
%Therefore, it follows from Theorem \ref{convprocess} that the sequence  $(H_p(k, t))_{t\in [\varepsilon,1-\varepsilon]}$ of random processes converges in law (in the space $C[\varepsilon,1-\varepsilon]$) to the process  $(H_{\X}(t))_{t\in [\varepsilon, 1-\varepsilon]}$.
%We now consider the following functional $$\widetilde{\ell}_{\varepsilon}(\phi)=\int_{\varepsilon}^{1-\varepsilon}\log (|\phi(t)|)\mathbbm{1}_{|\phi(t)|\ge \varepsilon}dt.$$
%Since this is a continuous functional on ${C}[\varepsilon,1-\varepsilon]$ we deduce that 
Fix $\varepsilon>0$ and record that $$
\lim_{p\to\infty}\frac{1}{p}\sum_{0\le k\le p-1}\int_{\varepsilon}^{1-\varepsilon}\log (|{H}_p(k,t)|)\mathbbm{1}_{|{H}_p(k,t)|\ge \varepsilon}dt=\int_{\varepsilon}^{1-\varepsilon}\mathbb{E}\big(\log (|{H}_{\mathbb{X}}(t)|)\mathbbm{1}_{|{H}_{\mathbb{X}}(t)|\ge \varepsilon}\big)dt.
$$
Furthermore, it follows from Lemmas \ref{lem: log-trunc} and \ref{lem: logtrunc-01} that 
$$\lim_{p\to\infty}\frac{1}{p}\sum_{0\le k\le p-1}\int_0^1\log(|{H}_p(k,t)|)dt=\int_{0}^{1}\mathbb{E}\big(\log(|{H}_{\mathbb{X}}(t)|)\big)dt+O\left(\varepsilon^{5/6}\right).$$
Letting $\varepsilon\to 0$ and adding $\int_0^1 \log|e(t)-1|dt -\log(2\pi)$ to both sides completes the proof of Theorem \ref{cor: Mahnorms}.

\end{proof}
\begin{rem}\label{rem:cst} We can effectively compute the constant $k_0$ as:
\begin{equation*}\label{computation-cst} 
k_0= \lim_{J \rightarrow \infty} \frac{1}{2^{J+1}} \sum_{\epsilon_j \in \{-1,1\}^{J+1}} \int_{0}^{1} \log \left\vert \sum_{\vert j\vert \leq J} \frac{\epsilon_j}{j-t}\right\vert dt.
\end{equation*} Indeed, for any $J\geq 1$, we can define the random approximate process
\begin{equation*}\label{def:approxprocess}
G_{\X}^{J}(t)= \sum_{m\in \mathbb{Z}, \vert m\vert \leq J} \frac{e(t)-1}{2\pi i (m-t)} \X(m), \hspace{2mm}t\in[0,1],
\end{equation*} and note that
\begin{equation*}\label{limit-cst}
\ex \left(\int_{0}^{1}\log \vert G_{\X}(t)\vert dt\right)= \lim_{J \rightarrow  \infty}\ex \left( \int_{0}^{1}\log \vert G_{\X}^{J}(t)\vert dt\right).
\end{equation*} Indeed, following our proof of Theorem \ref{convprocess} (using in particular the same computation made to derive \eqref{SecondMomentApprox1}), we can show that the sequence of random processes $(G_{\X}^{J}(t))_{t\in [0,1]}$ converges in law to the process $(G_{\X}(t))_{t \in [0,1]}$ in $C([0,1])$. It is then straightforward to see that the arguments used to prove Lemmas \ref{lem: log-trunc} and \ref{lem: logtrunc-01} apply mutatis mutandis to $H_{\X}^{J}=\frac{2\pi i}{e(t)-1} G_{\X}^{J}.$

%shows.used to show that  we have
%$\lim_{J \rightarrow + \infty} \mathbb{E} \lVert G_{\X}^{J}-G_{\X}\rVert_{\infty}^2=0$, as the same computation made to derive \eqref{SecondMomentApprox1} shows.
\end{rem}

% Taking $\varepsilon\to 0$ in \eqref{truncated-convergence} and applying Lemmas \ref{lem: log-trunc} together with the monotone convergence theorem we deduce
 %$$\lim_{p\to\infty}\frac{1}{p}\sum_{0\le k\le p-1}\int_{0}^1\log (|{G}_p(k,t)|)dt=\int_{0}^1\mathbb{E}\log (|{G}_{\mathbb{X}}(t)|dt:=\kappa.$$

\section*{Acknowledgements}
The authors would like to express their gratitude to Andrew Granville for extremely fruitful discussions about Fekete polynomials that, in particular, led to the proof of Theorem \ref{cor: Mahnorms}.
We are also grateful to Joseph Najnudel for valuable discussions.
MM  would like to thank the Max Planck Institute for
Mathematics, Bonn for the hospitality during his work on this project. MM also acknowledges support by the Austrian Science Fund (FWF), stand-alone project P 33043  ``Character sums, L-functions and applications'' and by the Ministero della Istruzione e della Ricerca ``Young Researchers Program Rita Levi Montalcini''. Part of this work was completed while YL was on a D\'el\'egation CNRS at the
IRL3457 CRM-CNRS in Montr\'eal. He  would like to thank the CNRS for its support and
the Centre de Recherches Math\'ematiques for its excellent working conditions. Finally, the authors would like to thank the Heilbronn Focused Research grant scheme for support.

\bibliographystyle{plain}
\bibliography{fekete}

\begin{thebibliography}{10}

\bibitem{BaMo}
R.~C. Baker and H.~L. Montgomery.
\newblock Oscillations of quadratic {$L$}-functions.
\newblock In {\em Analytic number theory ({A}llerton {P}ark, {IL}, 1989)},
  volume~85 of {\em Progr. Math.}, pages 23--40. Birkh\"{a}user Boston, Boston,
  MA, 1990.

\bibitem{flatexist}
P.~Balister, B.~Bollob{\'a}s, R.~Morris, J.~Sahasrabudhe, and M.~Tiba.
\newblock Flat {L}ittlewood polynomials exist.
\newblock {\em Ann. Math}, 192(3):977--1004, 2020.

\bibitem{BPW}
P.~Bateman, G.~Purdy, and S.~Wagstaff.
\newblock Some numerical results on {F}ekete polynomials.
\newblock {\em Math. Comput.}, 29:7--23, 1975.
\newblock Collection of articles dedicated to Derrick Henry Lehmer on the
  occasion of his seventieth birthday.

\bibitem{beck}
J.~Beck.
\newblock Flat polynomials on the unit circle---note on a problem of
  {L}ittlewood.
\newblock {\em Bull. London Math. Soc.}, 23(3):269--277, 1991.

\bibitem{BSh}
J.~Bell and I.~Shparlinski.
\newblock Power series approximations to {F}ekete polynomials.
\newblock {\em J. Approx. Theory}, 222:132--142, 2017.

\bibitem{B-N-R}
J.~Benatar, A.~Nishry, and B.~Rodgers.
\newblock Moments of polynomials with random multiplicative coefficients.
\newblock {\em Mathematika}, 68(1):191--216, 2022.

\bibitem{Bernstein}
S.~N. Bernstein.
\newblock Sur l’ordre de la meilleure approximation des fonctions continues
  par les polynômes de degré donné.
\newblock {\em Mémoires publiés par la Classe des Sciences de l’Académie
  de Belgique}, 4, 1912.

\bibitem{Bill}
P.~Billingsley.
\newblock {\em Convergence of probability measures}.
\newblock Wiley Series in Probability and Statistics: Probability and
  Statistics. John Wiley \& Sons, Inc., New York, second edition, 1999.
\newblock A Wiley-Interscience Publication.

\bibitem{Borweinexcursion}
P.~Borwein.
\newblock {\em Computational Excursions in Analysis and Number Theory}.
\newblock Springer, 2022.

\bibitem{BC}
P.~Borwein and K.~S. Choi.
\newblock Explicit merit factor formulae for {F}ekete and {T}uryn polynomials.
\newblock {\em Trans. Amer. Math. Soc.}, 354(1):219--234, 2002.

\bibitem{borweinBarker}
P.~Borwein, K.~S. Choi, and J.~Jankauskas.
\newblock On a class of polynomials related to {B}arker sequences.
\newblock {\em Proc. Amer. Math. Soc.}, 140(8):2613--2625, 2012.

\bibitem{borwein2001extremal}
P.~Borwein, K.~S. Choi, and S.~Yazdani.
\newblock An extremal property of {F}ekete polynomials.
\newblock {\em Proc. Amer. Soc.}, 129(1):19--27, 2001.

\bibitem{BC-merit}
P.~Borwein and S.~Choi.
\newblock Merit factors of polynomials formed by {J}acobi symbols.
\newblock {\em Canad. J. Math.}, 53(1):33--50, 2001.

\bibitem{Bo-Lo}
P.~Borwein and R.~Lockhart.
\newblock The expected {$L_p$} norm of random polynomials.
\newblock {\em Proc. Amer. Math. Soc.}, 129(5):1463--1472, 2001.

\bibitem{BBflat}
J.~Bourgain and E.~Bombieri.
\newblock On {K}ahane's ultraflat polynomials.
\newblock {\em J. Eur. Math. Soc.}, 11(3):627--703, 2009.

\bibitem{Aveg-Mah}
K.~S. Choi and T.~Erd\'{e}lyi.
\newblock Average {M}ahler's measure and {$L_p$} norms of {L}ittlewood
  polynomials.
\newblock {\em Proc. Amer. Math. Soc. Ser. B}, 1:105--120, 2014.

\bibitem{Chow}
S.~Chowla.
\newblock Note on a {D}irichlet {L}- function.
\newblock {\em Acta Arith.}, 1(1):113--114, 1936.

\bibitem{CGPS}
B.~Conrey, A.~Granville, B.~Poonen, and K.~Soundararajan.
\newblock Zeros of {F}ekete polynomials.
\newblock {\em Ann. Inst. Fourier (Grenoble)}, 50(3):865--889, 2000.

\bibitem{ConreySound}
J.~B. Conrey and K.~Soundararajan.
\newblock Real zeros of quadratic {D}irichlet {$L$}-functions.
\newblock {\em Invent. Math.}, 150(1):1--44, 2002.

\bibitem{E-Lpgrowth}
T.~Erd\'{e}lyi.
\newblock Upper bounds for the {$L_q$} norm of {F}ekete polynomials on subarcs.
\newblock {\em Acta Arith.}, 153(1):81--91, 2012.

\bibitem{Lower-Mah}
T.~Erd\'{e}lyi.
\newblock Improved lower bound for the {M}ahler measure of the {F}ekete
  polynomials.
\newblock {\em Constr. Approx.}, 48(2):283--299, 2018.

\bibitem{ErdMahler1}
T.~Erd\'{e}lyi.
\newblock The asymptotic value of the {M}ahler measure of the {R}udin-{S}hapiro
  polynomials.
\newblock {\em J. Anal. Math.}, 142(2):521--537, 2020.

\bibitem{Erdunitcircle}
T.~Erd\'{e}lyi.
\newblock On the oscillation of the modulus of the {R}udin-{S}hapiro
  polynomials on the unit circle.
\newblock {\em Mathematika}, 66(1):144--160, 2020.

\bibitem{ErdSurv}
T.~Erd\'{e}lyi.
\newblock Recent progress in the study of polynomials with constrained
  coefficients.
\newblock In {\em Trigonometric sums and their applications}, pages 29--69.
  Springer, Cham, 2020.

\bibitem{LE}
T.~Erd\'{e}lyi and D.~S. Lubinsky.
\newblock Large sieve inequalities via subharmonic methods and the {M}ahler
  measure of the {F}ekete polynomials.
\newblock {\em Canad. J. Math.}, 59(4):730--741, 2007.

\bibitem{Erdospb}
P.~Erd\H{o}s.
\newblock Some unsolved problems.
\newblock {\em Michigan Math. J.}, 4:291--300, 1957.

\bibitem{FeketePolya}
M.~Fekete and G.~P{\'o}lya.
\newblock {\"U}ber ein problem von {L}aguerre.
\newblock {\em Rendiconti del Circolo Matematico di Palermo (1884-1940)},
  34(1):89--120, 1912.

\bibitem{GS-Lp}
C.~G\"{u}nther and K-U. Schmidt.
\newblock {$L^q$} norms of {F}ekete and related polynomials.
\newblock {\em Canad. J. Math.}, 69(4):807--825, 2017.

\bibitem{Harp-moving-character}
A.~Harper.
\newblock A note on character sums over short moving intervals.
\newblock {\em https://arxiv.org/abs/2203.09448}, 2022.

\bibitem{Heilbr}
H.~Heilbronn.
\newblock On real characters.
\newblock {\em Acta Arith.}, 2:212--213, 1936.

\bibitem{Hoholdtmerit}
T.~Hoholdt and H.~E. Jensen.
\newblock Determination of the merit factor of {L}egendre sequences.
\newblock {\em IEEE Transactions on Information Theory}, 34(1):161--164, 1988.

\bibitem{Rud-Hugh}
C.~P. Hughes and Z.~Rudnick.
\newblock On the distribution of lattice points in thin annuli.
\newblock {\em Int. Math. Res. Not.}, (13):637--658, 2004.

\bibitem{IwKo}
H.~Iwaniec and E.~Kowalski.
\newblock {\em Analytic number theory}, volume~53 of {\em American Mathematical
  Society Colloquium Publications}.
\newblock American Mathematical Society, Providence, RI, 2004.

\bibitem{JKS}
J.~Jedwab, D.~J. Katz, and K-U. Schmidt.
\newblock Advances in the merit factor problem for binary sequences.
\newblock {\em Journal of Combinatorial Theory, Series A}, 120(4):882--906,
  2013.

\bibitem{jedwab2013littlewood}
J.~Jedwab, D.~J Katz, and K-U. Schmidt.
\newblock Littlewood polynomials with small ${L}^4$ norm.
\newblock {\em Advances in Mathematics}, 241:127--136, 2013.

\bibitem{JHH}
J.~M. Jensen, H.~E. Jensen, and T.~Hoholdt.
\newblock The merit factor of binary sequences related to difference sets.
\newblock {\em IEEE Transactions on Information Theory}, 37(3):617--626, 1991.

\bibitem{Kahane}
J-P. Kahane.
\newblock Sur les polyn{\^o}mes {\`a} coefficients unimodulaires.
\newblock {\em Bull. Lond. Math. Soc.}, 12(5):321--342, 1980.

\bibitem{Kon-Sch}
S.~V. Konyagin and W.~Schlag.
\newblock Lower bounds for the absolute value of random polynomials on a
  neighborhood of the unit circle.
\newblock {\em Trans. Amer. Math. Soc.}, 351(12):4963--4980, 1999.

\bibitem{KoBook}
E.~Kowalski.
\newblock {\em An introduction to probabilistic number theory}, volume 192 of
  {\em Cambridge Studies in Advanced Mathematics}.
\newblock Cambridge University Press, Cambridge, 2021.

\bibitem{KowSaw}
E.~Kowalski and W.~Sawin.
\newblock Kloosterman paths and the shape of exponential sums.
\newblock {\em Compositio Math.}, 152(7):1489--1516, 2016.

\bibitem{Littlewood-Mahler}
J.~E. Littlewood.
\newblock The real zeros and value distributions of real trigonometrical
  polynomials.
\newblock {\em J. London Math. Soc.}, 41:336--342, 1966.

\bibitem{littlewood1968some}
J.~E. Littlewood.
\newblock {\em Some problems in real and complex analysis}.
\newblock DC Heath, 1968.

\bibitem{mont}
H.~L. Montgomery.
\newblock {\em Topics in multiplicative number theory}.
\newblock Lecture Notes in Mathematics, Vol. 227. Springer-Verlag, Berlin-New
  York, 1971.

\bibitem{Mont-large}
H.~L. Montgomery.
\newblock An exponential polynomial formed with the {L}egendre symbol.
\newblock {\em Acta Arith.}, 37:375--380, 1980.

\bibitem{Mont-Litt}
H.~L Montgomery.
\newblock Littlewood polynomials.
\newblock In {\em Analytic Number Theory, Modular Forms and q-Hypergeometric
  Series: In Honor of Krishna Alladi's 60th Birthday, University of Florida,
  Gainesville, March 2016}, pages 533--553. Springer, 2018.

\bibitem{NNS}
F.~Nazarov, A.~Nishry, and M.~Sodin.
\newblock Log-integrability of {R}ademacher {F}ourier series, with applications
  to random analytic functions.
\newblock {\em Algebra i Analiz}, 25(3):147--184, 2013.

\bibitem{newman}
D.~J. Newman.
\newblock An {L$^1$} extremal problem for polynomials.
\newblock {\em Proc. Amer. Math. Soc.}, 16(6):1287--1290, 1965.

\bibitem{odlyzko}
A.~Odlyzko.
\newblock Search for ultraflat polynomials with plus and minus one
  coefficients.
\newblock In {\em Connections in discrete mathematics}, pages 39--55, 2018.

\bibitem{Rodgers}
B.~Rodgers.
\newblock On the distribution of {R}udin-{S}hapiro polynomials and lacunary
  walks on {$SU(2)$}.
\newblock {\em Adv. Math.}, 320:993--1008, 2017.

\bibitem{saffariRudin}
B.~Saffari.
\newblock Une fonction extr{\'e}male li{\'e}e {\`a} la suite de
  {R}udin--{S}hapiro.
\newblock {\em CR Acad. Sci. Paris S{\'e}r. I Math}, 303(4):97--100, 1986.

\bibitem{S-Z}
R.~Salem and A.~Zygmund.
\newblock Some properties of trigonometric series whose terms have random
  signs.
\newblock {\em Acta Math.}, 91:245--301, 1954.

\end{thebibliography}
\nocite{KowSaw}

\end{document}